\documentclass[submission,copyright,creativecommons,sharealike]{eptcs}
\providecommand{\event}{ACT 2024} %
\providecommand{\thisvolume}[1]{this volume of EPTCS, Open Publishing Association}
\usepackage{breakurl}             %
\usepackage{underscore}           %

\usepackage[utf8]{inputenc}
\usepackage[T1]{fontenc} %

\usepackage{microtype}
\usepackage{float}

\usepackage{amsfonts}
\usepackage{amssymb}
\usepackage{amsthm}
\usepackage{amsmath}
\usepackage{caption}
\usepackage{etoolbox}
\usepackage{stmaryrd}
\usepackage{ifthen}
\usepackage{suffix}
\usepackage{verbatim} %

\usepackage{xfrac}

\usepackage{bm}
\usepackage{physics}
\usepackage[safe]{tipa} %
\usepackage{mathtools} %
\usepackage{color}
\usepackage{wrapfig}
\usepackage{tikz}
\usepackage{tikzit}
\usetikzlibrary{cd, babel}

\usepackage[framemethod=tikz]{mdframed}

\newcommand*{\secref}[1]{\S\ref{#1}}

\usepackage[backend=biber,style=numeric-comp,date=short,maxbibnames=99]{biblatex}
\appto{\bibsetup}{\sloppy}

\newcommand*{\citep}[1]{\parencite{#1}}
\newcommand*{\citet}[1]{\textcite{#1}}

\usepackage[capitalize]{cleveref} %

\usepackage{graphicx}

\usepackage{pict2e}

\makeatletter
\newcommand{\adjunction}{\@ifstar\named@adjunction\normal@adjunction}
\newcommand{\normal@adjunction}[4]{%
  #1\colon #2%
  \mathrel{\vcenter{%
    \offinterlineskip\m@th
    \ialign{%
      \hfil$##$\hfil\cr
      \longrightharpoonup\cr
      \noalign{\kern-.3ex}
      \smallbot\cr
      \longleftharpoondown\cr
    }%
  }}%
  #3 \noloc #4%
}
\newcommand{\named@adjunction}[4]{%
  #2%
  \mathrel{\vcenter{%
    \offinterlineskip\m@th
    \ialign{%
      \hfil$##$\hfil\cr
      \scriptstyle#1\cr
      \noalign{\kern.1ex}
      \longrightharpoonup\cr
      \noalign{\kern-.3ex}
      \smallbot\cr
      \longleftharpoondown\cr
      \scriptstyle#4\cr
    }%
  }}%
  #3%
}
\newcommand{\longrightharpoonup}{\relbar\joinrel\rightharpoonup}
\newcommand{\longleftharpoondown}{\leftharpoondown\joinrel\relbar}
\newcommand\noloc{%
  \nobreak
  \mspace{6mu plus 1mu}
  {:}
  \nonscript\mkern-\thinmuskip
  \mathpunct{}
  \mspace{2mu}
}
\newcommand{\smallbot}{%
  \begingroup\setlength\unitlength{.15em}%
  \begin{picture}(1,1)
  \roundcap
  \polyline(0,0)(1,0)
  \polyline(0.5,0)(0.5,1)
  \end{picture}%
  \endgroup
}
\makeatother

\let\op=\relax

\def\op{\ensuremath{^{\,\mathrm{op}}}}

\newcommand{\Kl}{\mathcal{K}\mspace{-2mu}\ell}

\DeclareMathOperator*{\E}{\mathbb{E}}

\renewcommand{\d}{\mathrm{d}}

\newcommand{\Ba}{{\mathcal{B}}}
\newcommand{\Ca}{{\mathcal{C}}}

\newcommand{\Ea}{{\mathcal{E}}}

\newcommand{\Ka}{{\mathcal{K}}}

\newcommand{\Rb}{{\mathbb{R}}}

\newcommand{\xto}[2][]{\xrightarrow[#1]{#2}}

\makeatletter
\newcommand{\mathoverlap}[2]{\mathpalette\mathoverlap@{{#1}{#2}}}
\newcommand{\mathoverlap@}[2]{\mathoverlap@@{#1}#2}
\newcommand{\mathoverlap@@}[3]{\ooalign{$\m@th#1#2$\crcr\hidewidth$\m@th#1#3$\hidewidth}}
\makeatother

\makeatletter
\providecommand*{\xmapstofill@}{%
  \arrowfill@{\mapstochar\relbar}\relbar\rightarrow
}
\providecommand*{\xmapsto}[2][]{%
  \ext@arrow 0395\xmapstofill@{#1}{#2}%
}

\makeatletter
\def\slashedarrowfill@#1#2#3#4#5{%
  $\m@th\thickmuskip0mu\medmuskip\thickmuskip\thinmuskip\thickmuskip
   \relax#5#1\mkern-7mu%
   \cleaders\hbox{$#5\mkern-2mu#2\mkern-2mu$}\hfill
   \mathclap{#3}\mathclap{#2}%
   \cleaders\hbox{$#5\mkern-2mu#2\mkern-2mu$}\hfill
   \mkern-7mu#4$%
}
\def\rightslashedarrowfill@{%
  \slashedarrowfill@\relbar\relbar\mapstochar\rightarrow}
\newcommand\xslashedrightarrow[2][]{%
  \ext@arrow 0055{\rightslashedarrowfill@}{#1}{#2}}
\makeatother

\theoremstyle{definition}
\newtheorem{defn}{Definition}[section]
\newtheorem{notation}[defn]{Notation}

\newtheorem{rmk}[defn]{Remark}

\newtheorem*{rmk*}{Remark}

\newtheorem{prop}[defn]{Proposition}
\newtheorem{prop*}{Proposition}

\newtheorem{lemma}[defn]{Lemma}
\newtheorem{thm}[defn]{Theorem}
\newtheorem{cor}[defn]{Corollary}

\newtheorem*{thm*}{Theorem}
\newtheorem*{cor*}{Corollary}

\theoremstyle{remark}

\definecolor{darkblue}{rgb}{0,0,0.7}

\usepackage{bussproofs}
\usepackage{enumitem}
\usepackage{extarrows}

\tikzstyle{dot}=[inner sep=0.0mm, outer sep=0.0mm, minimum size=1mm, draw, shape=circle]
\tikzstyle{dot-5mm}=[dot, minimum size=5mm]
\tikzstyle{dot-1cm}=[dot, minimum size=1cm]
\tikzstyle{wcopy}=[dot, fill=white, scale=2.0]
\tikzstyle{bcopy}=[dot, fill=black, scale=2.0]
\tikzstyle{box}=[fill=white, draw=black, shape=rectangle]
\tikzstyle{box-5mm}=[box, minimum size=5mm, shape aspect=1]
\tikzstyle{box-7mm}=[box, minimum size=7mm, shape aspect=1]
\tikzstyle{box-1cm}=[box, minimum size=10mm, shape aspect=1]
\tikzstyle{dia}=[fill=white, draw=black, shape=diamond]
\tikzstyle{dia-5mm}=[dia, minimum size=5mm, shape aspect=1]
\tikzstyle{dia-7mm}=[dia, minimum size=7mm, shape aspect=1]
\tikzstyle{effect}=[regular polygon, regular polygon sides=3, draw]
\tikzstyle{state0}=[regular polygon, regular polygon sides=3, draw, shape border rotate=0]
\tikzstyle{state90}=[regular polygon, regular polygon sides=3, draw, shape border rotate=90]
\tikzstyle{state180}=[regular polygon, regular polygon sides=3, draw, shape border rotate=180]
\tikzstyle{state270}=[regular polygon, regular polygon sides=3, draw, shape border rotate=270]
\tikzstyle{scalar}=[diamond, draw, inner sep=1pt]
\tikzstyle{ground0}=[my ground, draw, inner sep=0pt, minimum width=4.2pt, minimum height=11.2pt, anchor=input, rotate=0]
\tikzstyle{ground90}=[my ground, draw, inner sep=0pt, minimum width=4.2pt, minimum height=11.2pt, anchor=input, rotate=90]

\input{strings.tikzdefs}

\ifexternalizetikz\tikzexternaldisable\fi
\newsavebox\sbground
\savebox\sbground{%
  \begin{tikzpicture}[baseline=0pt]
    \draw (0,-.1ex) to (0,.85ex)
    node[ground IEC,draw,anchor=input,inner sep=0pt,
    minimum width=3.15pt,minimum height=8.4pt,rotate=90] {};
  \end{tikzpicture}%
}

\newsavebox\sbcopy
\savebox\sbcopy{%
  \begin{tikzpicture}[baseline=0pt]
    \node[wcopy,scale=0.7] (a) at (0,3.8pt) {};
    \draw (a) -- +(-90:.21);
    \draw (a) -- +(45:.21);
    \draw (a) -- +(135:.21);
  \end{tikzpicture}}

\ifexternalizetikz\tikzexternalenable\fi

\newsavebox\bsbcopy
\savebox\bsbcopy{%
  \begin{tikzpicture}[baseline=0pt]
    \node[bcopy,scale=0.7] (a) at (0,3.8pt) {};
    \draw (a) -- +(-90:.21);
    \draw (a) -- +(45:.21);
    \draw (a) -- +(135:.21);
  \end{tikzpicture}}

\ifexternalizetikz\tikzexternalenable\fi

\usepackage{mathabx}
\usepackage{quiver}

\def\Ca{\mathcal{C}}
\def\Cat{\mathbf{Cat}}

\def\CCospan{\mathbb{C}\mathbf{ospan}}
\def\Comon{\mathbf{Comon}}
\def\FinSet{\mathbf{FinSet}}

\def\Meas{\mathbf{Meas}}
\def\Mod{\mathbf{Mod}}
\def\QBS{\mathbf{QBS}}

\def\sfKrn{\mathbf{sfKrn}}
\def\SSpan{\mathbb{S}\mathbf{pan}}
\def\Set{\mathbf{Set}}
\def\unif{\upsilon}

\def\FG{\Bbb{FG}}

\def\Ibb{\mathbb{I}}
\def\Nbb{\mathbb{N}}
\def\Rbb{\mathbb{R}}
\def\Sbb{\Bbb{S}}

\def\id{\mathsf{id}}
\def\disc{\mathsf{disc}\,}

\newcommand{\ovln}[1]{\overline{#1}}

\def\d{\mathop{}\!\mathrm{d}}
\def\kto{\rightsquigarrow}

\newcommand{\xkto}[1]{\overset{#1}{\kto}}
\def\xto{\xrightarrow}

\author{Toby St Clere Smithe
  \institute{VERSES Research Lab}
  \institute{Topos Institute}
  \email{act@tsmithe.net}
}
\def\authorrunning{T. St Clere Smithe}

\hypersetup{
  pdftitle={Copy-composition for Probabilistic Graphical Models},
  pdfauthor={Toby St Clere Smithe}
}

\date{12 June 2024}

\title{Copy-composition for Probabilistic Graphical Models\footnote{
  Some of this content appeared first on the author's website, such as at \url{https://tsmithe.net/p/factor-graphs.html}; however, none has appeared elsewhere.}}
\def\titlerunning{Copy-composition for Probabilistic Graphical Models}

\begin{document}

\maketitle

\begin{abstract}
  In probabilistic modelling, joint distributions are often of more interest than their marginals, but the standard composition of stochastic channels is defined by marginalization.
  Recently, the notion of `copy-composition' was introduced in order to circumvent this problem and express the chain rule of the relative entropy fibrationally, but while that goal was achieved, copy-composition lacked a satisfactory origin story.
  Here, we supply such a story for two standard probabilistic tools: directed and undirected graphical models.
  We explain that (directed) Bayesian networks may be understood as ``stochastic terms'' of product type, in which context copy-composition amounts to a pull-push operation.
  Likewise, we show that (undirected) factor graphs compose by copy-composition.
  In each case, our construction yields a double fibration of decorated (co)spans.
  Along the way, we introduce a useful bifibration of measure kernels, to provide semantics for the notion of stochastic term, which allows us to generalize probabilistic modelling from product to dependent types.
\end{abstract}

\section{Introduction}

Recent work has formalized the ``chain rule'' of the relative entropy (\textit{a.k.a.} Kullback-Leibler divergence) as a section of a fibration of statistical games over Bayesian lenses \parencite{Smithe2023Approximate}.
Curiously, however, the base category of this fibration was not a category of stochastic channels (measure kernels), as one might expect, but instead a bicategory of \textit{copy-composite} channels.
The need for this stemmed from the Chapman-Kolmogorov rule for the composition of kernels, which involves marginalization; and marginalization does not commute with the logarithm defining the relative entropy.
Copy-composition defers this marginalization (indeed, the Chapman-Kolmogorov rule is precisely the marginalization of the intermediate copy), so there is no problem of commutativity, and the chain rule thus obtains.

Although this yields an elegant fibrational structure, copy-composition is not strictly unital, owing to the extra copy produced when copy-composing with the identity.
This means that the base of the fibration is necessarily a bicategory.
Moreover, the construction of this bicategory \parencite[\S2]{Smithe2023Approximate} was largely \textit{ad hoc}, based on the similarity to a $\mathbf{CoPara}$ construction \parencite[Remark 4]{Capucci2022Foundations}.
This was despite the evident importance of copy-composition to the statistical problems of interest there: copy-composition constructs \textit{joint distributions}, which are often of more interest in applications than the composite channels that would yield their marginals.
This can be seen not only via the relative entropy, but also in the context of Bayesian networks, which describe the factorization of a joint distribution according to a graphical structure.
\textcite{Fong2013Causal} studied Bayesian networks from a categorical perspective, making extensive use of comonoid (`copy-discard' \parencite{Cho2017Disintegration}) structure to construct the corresponding joint distributions; see for example the proof of his Theorem 4.5.

In the present submission, we show how copy-composition emerges naturally (but distinctly) from the compositional structure of directed and undirected probabilistic graphical models, thereby replacing the earlier arbitrariness with something structurally meaningful.
In the directed case, the key observation is that a Bayesian network may be understood as a ``stochastic term'' of some displayed type (which is invariably a product), equipped with a distribution (the `prior') over its context.
Identifying types with bundles and terms with sections, this means understanding the factors of a Bayesian network as \textit{stochastic sections}; copy-composition then emerges via a pull-push operation.
In the undirected case, one observes that the `factors' of a factor graph are costates\footnote{
A \textit{state} in a monoidal category is a morphism out of the monoidal unit; a \textit{costate} is a morphism into the monoidal unit.}
in a copy-discard category, and that copy-composition amounts to a kind of gluing.
In both cases, copy-composition appears as the horizontal composition operation of a pseudo double category, thereby extending the earlier bicategorical story.
These double categories\footnote{
We take ``double category'' henceforth to mean \textit{pseudo double category}, oriented such that composition is strict in the vertical category and weak in the horizontal.}
are in fact double fibrations, following the generalized decorated (co)span construction of \textcite{Patterson2023Structured}: in the directed case, one decorates spans of display maps with sections; in the undirected case, one decorates cospans with factors.

\section{Directed models: pull-push stochastic sections}

A \textit{Bayesian network} is a distribution (or measure) on a product space that factorizes into a product of conditional distributions (or measure kernels) that matches the structure of a directed acyclic graph \parencite[Definition 3.7]{Fong2013Causal}.
Interpreting this idea in a category of stochastic channels (such as measure kernels between measurable spaces), \textcite[Theorem 4.5]{Fong2013Causal} showed that any Bayesian network can be written as a state $1\to \otimes_j\,X_j$ that factorizes as the sequential composition of morphisms of the form
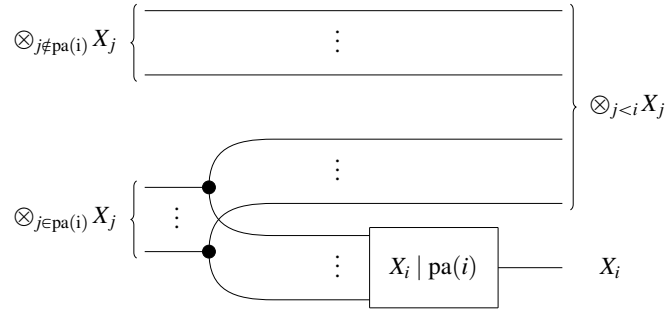
\[\scalebox{0.85}{\begin{tikzpicture}
	\begin{pgfonlayer}{nodelayer}
		\node [style=none] (0) at (2, 3) {};
		\node [style=none] (1) at (-1, 3) {};
		\node [style=none] (2) at (2, 1) {};
		\node [style=none] (3) at (-1, 1) {};
		\node [style=none] (4) at (1, 2.25) {$\vdots$};
		\node [style=none] (5) at (2, 3.25) {};
		\node [style=none] (6) at (2, 0.75) {};
		\node [style=none] (7) at (6, 0.75) {};
		\node [style=none] (8) at (6, 3.25) {};
		\node [style=none] (9) at (3.75, 2) {};
		\node [style=none] (10) at (4, 2) {$X_i \mid \mathrm{pa}(i)$};
		\node [style=none] (11) at (-1, 4) {};
		\node [style=none] (12) at (-1, 6) {};
		\node [style=none] (13) at (8, 2) {};
		\node [style=none] (14) at (8, 4) {};
		\node [style=none] (15) at (8, 6) {};
		\node [style=none] (16) at (6, 2) {};
		\node [style=none] (17) at (1, 5.25) {$\vdots$};
		\node [style=bcopy] (19) at (-3, 4.5) {};
		\node [style=bcopy] (20) at (-3, 2.5) {};
		\node [style=none] (21) at (-5, 2.5) {};
		\node [style=none] (22) at (-5, 4.5) {};
		\node [style=none] (23) at (-4, 3.75) {$\vdots$};
		\node [style=none] (24) at (8, 8) {};
		\node [style=none] (25) at (8, 10) {};
		\node [style=none] (26) at (-5, 8) {};
		\node [style=none] (27) at (-5, 10) {};
		\node [style=none] (28) at (1, 9.25) {$\vdots$};
		\node [style=none] (29) at (-5.25, 4.75) {};
		\node [style=none] (30) at (-5.25, 2.25) {};
		\node [style=none] (31) at (-5.25, 7.75) {};
		\node [style=none] (32) at (-5.25, 10.25) {};
		\node [style=none] (33) at (8.25, 10.25) {};
		\node [style=none] (34) at (8.25, 3.75) {};
		\node [style=none] (35) at (9.5, 2) {$X_i$};
	\end{pgfonlayer}
	\begin{pgfonlayer}{edgelayer}
		\draw (0.center)
			 to (1.center)
			 to [in=-90, out=180, looseness=1.25] (19.center)
			 to [in=-180, out=90, looseness=1.25] (12.center)
			 to (15.center);
		\draw (2.center)
			 to (3.center)
			 to [in=-90, out=180, looseness=1.25] (20.center)
			 to [in=-180, out=90, looseness=1.25] (11.center)
			 to (14.center);
		\draw (5.center) to (6.center);
		\draw (8.center) to (7.center);
		\draw (5.center) to (8.center);
		\draw (6.center) to (7.center);
		\draw (16.center) to (13.center);
		\draw (22.center) to (19);
		\draw (21.center) to (20);
		\draw (27.center) to (25.center);
		\draw (26.center) to (24.center);
		\draw [decorate, decoration = {brace}] (30) -- node[left=5pt] {${\otimes_{j\in\mathrm{pa(i)}}\,X_j}$}  (29);
		\draw [decorate, decoration = {brace}] (31) -- node[left=5pt] {${\otimes_{j\notin\mathrm{pa(i)}}\,X_j}$}  (32);
		\draw [decorate, decoration = {brace}] (33) -- node[right=5pt] {${\otimes_{j<i}\,X_j}$}  (34);
	\end{pgfonlayer}
\end{tikzpicture}}\]
where $\mathrm{pa}(i)$ denotes the set of parents of the vertex $i$, and where $j<i$ iff there is no directed path from $i$ to $j$ (Fong calls this the `ancestral' order on the vertices).

Each factor $X_i\mid \mathrm{pa}(i)$ only appears in the composite diagram as precomposed by a copier applied to its parents, and tensored with the identities of its grandparents.
Interpreting the analogous string diagram
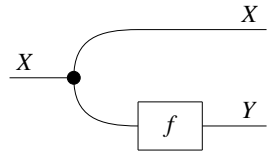
\[\scalebox{0.85}{\begin{tikzpicture}
	\begin{pgfonlayer}{nodelayer}
		\node [style=none] (0) at (1, -1.5) {$f$};
		\node [style=none] (1) at (0, -0.75) {};
		\node [style=none] (2) at (2, -0.75) {};
		\node [style=none] (3) at (0, -2.25) {};
		\node [style=none] (4) at (2, -2.25) {};
		\node [style=none] (5) at (0, -1.5) {};
		\node [style=none] (6) at (2, -1.5) {};
		\node [style=none] (7) at (4, -1.5) {};
		\node [style=none] (8) at (4, 1.5) {};
		\node [style=bcopy] (9) at (-2, 0) {};
		\node [style=none] (10) at (0, 1.5) {};
		\node [style=none] (11) at (-4, 0) {};
		\node [style=none] (12) at (-3.5, 0.5) {$X$};
		\node [style=none] (13) at (3.5, 2) {$X$};
		\node [style=none] (14) at (3.5, -1) {$Y$};
	\end{pgfonlayer}
	\begin{pgfonlayer}{edgelayer}
		\draw (11.center) to (9);
		\draw (8.center)
			 to (10.center)
			 to [in=90, out=180, looseness=1.25] (9.center)
			 to [in=-180, out=-90, looseness=1.25] (5.center);
		\draw (6.center) to (7.center);
		\draw (1.center) to (2.center);
		\draw (3.center) to (4.center);
		\draw (1.center) to (3.center);
		\draw (2.center) to (4.center);
	\end{pgfonlayer}
\end{tikzpicture}}\]
in $\Set$, we see that it corresponds to the graph of the function $f:X\to Y$, which is a section of the product projection $X\times Y\to X$; in other words, a term of $X\times Y$ over the context $X$.
Moreover, `tensoring' this diagram with extra objects $W$ corresponds to pulling back the section along the projection $W\times X\to X$; in other words, extending the context by $W$.
Thus, whereas \textcite{Fong2013Causal} constructed the preceding diagram by hand, with the right categorical setting we will be able to interpret it as resulting from standard type-theoretic operations.
Our first task will therefore be to establish a bifibration of measure kernels in which these operations may be interpreted.
Then, we will construct a double category whose horizontal composition implements these operations (and thus copy-composition) via `pull-push' in the bifibration.

\subsection{A bifibration of quasi-Borel kernels}

The classic categorical semantics for dependent type theory is in fibrations \parencite{Jacobs1999Categorical,nLabauthors2024Categorical}: the base category is interpreted as a category of `contexts'; the fibre over an object interprets the types in the corresponding context; elements in a fibre interpret terms; and pullback (base change) interprets substitution.
The classic example is the codomain fibration of a finitely complete category, in which bundles `display' types, and their sections are corresponding terms.
The category $\Meas$ of measurable spaces and measurable functions between them is a finitely complete category, and so this story may be told there.
But measurable functions are still deterministic; they are not kernels.

One may hope to replace the functions in the fibres by kernels and still obtain a fibration, but the image of a measurable subset under an arbitrary measurable map may not again be measurable\footnote{
This is guaranteed when the map is a continuous injection, by the Lusin-Souslin theorem \parencite[{Theorem 15.1}]{Kechris1995Classical}, but not in general.
Thanks to Ohad Kammar for pointing this out to me.},
which means that we cannot generally pull back traditional kernels.
Still, all is not lost, for we may turn our attention to quasi-Borel spaces \parencite{Heunen2017Convenient}, where the basic notion of \textit{measurable subset} is replaced by that of \textit{random element}, and the notions of measure and kernel are adjusted accordingly.
Additionally, there is an adjunction between $\Meas$ and the category $\QBS$ of quasi-Borel spaces that restricts to an equivalence on standard Borel spaces \parencite[Prop. 15]{Heunen2017Convenient}, so we do not stray too far from the familiar.

\begin{defn}[$\QBS$]
  Let $\Ibb$ denote the unit interval $[0,1]$.
A \textit{quasi-Borel space} is a pair $(X,M_X)$ of a set $X$ along with a subset $M_X \subseteq \Set(\Ibb,X)$ of functions $\Ibb\to X$ (\textit{random elements} of $X$), satisfying the following axioms:
\begin{enumerate}
\item (\textit{constants}) if $\rho:\Ibb\to X$ is a constant function, then $\rho\in M_X$;
\item (\textit{closure under measurable precomposition}) if $f:\Ibb\to\Ibb$ is measurable (\textit{i.e.}, $f\in\Meas(\Ibb,\Ibb)$) and $\rho\in M_X$, then $\rho\circ f \in M_X$;
\item (\textit{sheaf}) if $\Ibb$ is partitioned as $\Ibb = \coprod_{i\in\Nbb} S_i$ with each $S_i\in\Sigma_{\Ibb}$, and if, for each $i$, we have $\alpha_i\in M_X$, then $\coprod_{i\in\Nbb}\alpha_i \in M_X$.
\end{enumerate}
A morphism of quasi-Borel spaces $f:(X,M_X)\to (Y,M_Y)$ is a function $f \in \Set(X,Y)$ such that $\rho\in M_X$ entails $f\circ\rho\in M_Y$.
These morphisms compose as functions, yielding a category, $\QBS$.
\end{defn}

\begin{rmk}
  The definition of quasi-Borel space given by \textcite{Heunen2017Convenient} uses $\Rbb$ as the source of randomness, rather than $\Ibb$.
  But $\Rbb$ and $\Ibb$ equipped with their standard Borel structure are measurably isomorphic, so nothing is lost by this change, and it simplifies the presentation that follows.
\end{rmk}

Any quasi-Borel space $(X,M_X)$ can be made into a measurable space $(X,\Sigma_X)$ whose $\sigma$-algebra is given by $\Sigma_X = \{U\subseteq X \mid \forall \alpha \in M_X . \alpha^{-1}(U) \in \Sigma_\Ibb \}$, where $\Sigma_\Ibb$ is the Borel $\sigma$-algebra associated to $\Ibb$; this is Prop. 14 of \textcite{Heunen2017Convenient}.
By equipping the noise source $\Ibb$ with a measure, the random elements become random variables.
We can then push forward the noise to obtain measures on (the measurable space associated to) each quasi-Borel space.
Thus we can consider the set of probability measures on a quasi-Borel space to be the set of measures induced in this way.

\begin{defn}
  Let $\unif$ denote the uniform measure on $\Ibb$.
  A \textit{probability measure} on a quasi-Borel space $(X,M_X)$ is an equivalence class of elements of $M_X$, quotiented by equivalence of measure under pushforward of $\unif$.
  Let $P(X)$ denote the set of such measures on $(X,M_X)$, which we can write as $\{\alpha_*\unif \in G(X,\Sigma_X) \mid \alpha\in M_X\}$, where $\alpha_*\unif$ denotes the pushforward measure $U\mapsto \unif\bigl(\alpha^{-1}(U)\bigr)$.
  It has a quasi-Borel structure given by $M_{P(X)} = \{\beta:\Ibb\to P(X) \mid \exists \alpha:\Ibb\to M_X . \forall r\in \Ibb . \beta(r) = \alpha(r)_*\unif \}$.
\end{defn}

$P$ is then made into a functor $\QBS\to\QBS$ by $P(f) : \alpha_*\unif \mapsto (f\circ\alpha)_*\unif$, and finally upgraded to a monad by analogy with the Giry monad $G:\Meas\to\Meas$; this is the content of §V.D of \textcite{Heunen2017Convenient}.
(Equivalently, $P$ is a quotient of the $\Ibb$-continuation monad by equivalence of measure.)
A kernel between quasi-Borel spaces is then a morphism in the Kleisli category $\Kl(P)$, which we denote with squiggly arrows.
The standard inclusion of $\QBS$ into $\Kl(P)$, given by post-composing with the monad unit, maps each function $f \in \QBS(X,Y)$ to a deterministic kernel $\delta_f:X\kto Y$.

We will use $\Kl(P)$ to define a bifibration $\Ka$ over $\QBS$, such that the objects of each fibre $\Ka_B$ will be (deterministic) functions into $B$, and the morphisms $(E,\pi)\kto(E',\pi')$ in $\Ka_B$ will be kernels $E\kto E'$ ``fibrewise over $B$''.

\begin{defn}[The fibres of $\Ka$]
  Let $B$ be a quasi-Borel space; we define a corresponding category $\Ka_B$.
  Its objects are pairs $(E,p)$ of a quasi-Borel space $E$ and a quasi-Borel function $p:E\to B$.
  A morphism $k:(E,p)\to(E',p')$ is a kernel $k:E\kto E'$ such that $\delta_{p'}\circ k = \delta_p$; we call these morphisms \textit{kernels over} $B$.
  Composition of morphisms is given by the Kleisli composition of the corresponding kernels; identities are given by identity kernels.
  (It is easy to see that this yields a well-defined category.)
\end{defn}

To see that the morphisms of $\Ka_B$ really are kernels fibrewise over $B$, let us introduce some notation.

\begin{notation}[Fibre notation]
  Suppose $p:E\to B$ is a function and $b:J\to B$ is a generalized element of $B$ (another function).
  The fibre of $p$ over $b$ is given by the pullback of $p$ along $b$:
  \[\begin{tikzcd}[cramped,sep=scriptsize]
    {p[b]} & E \\
    J & B
    \arrow["{\pi_E}", from=1-1, to=1-2]
    \arrow["p", from=1-2, to=2-2]
    \arrow["b"', from=2-1, to=2-2]
    \arrow["{b^*p}"', from=1-1, to=2-1]
    \arrow["\lrcorner"{anchor=center, pos=0.125}, draw=none, from=1-1, to=2-2]
  \end{tikzcd}\]
  We denote the fibre by $p[b]\xto{b^*p}J$.
\end{notation}

\begin{prop}[Morphisms in $\Ka_B$ are fibrewise kernels] \label{prop:fib-krn}
  Suppose $k:(E,p)\kto(E,p')$ is a kernel over $B$, and $b:J\to B$ is a generalized element of $B$.
  Then $k$ restricts to a kernel $k[b]:p[b]\kto p'[b]$.
  \begin{proof}
    $\delta_p$ maps each $x\in E$ to the equivalence class generated by the constant random element $r\mapsto p(x)$.
    So the condition $\delta_{p'}\circ k = \delta_p$ means that $p(x) = p'(\ovln{k(x)}(r))$, for all $x\in E$, $r\in\Ibb$, and elements $\ovln{k(x)}$ of the equivalence class corresponding to $k(x)$.
    Let $\pi_E:p[b]\to E$ be the projection as in the preceding pullback square; and define $k[b]$ by mapping $x\in p[b]$ to $k(\pi_E(x))$.
    Let $j = b^*p(x)$.
    By assumption, we have $p(\pi_E(x)) = b(j)$.
    We need to verify that for all $r\in\Ibb$, $\ovln{k(\pi_E(x))}(r) \in p'[b]$.
    By definition, we have $p'(\ovln{k(\pi_E(x))}(r)) = p(\pi_E(x)) = b(j)$.
    So $\ovln{k(\pi_E(x))}(r)$ is in the fibre of $p'$ over $b(j)$, which is contained within $p'[b]$.
    So $\ovln{k(\pi_E(x))}(r) \in p'[b]$ as required.
  \end{proof}
\end{prop}

It will be very useful to extend this fibrewise-restriction to composite kernels; fortunately, it commutes with composition.
For the proof, see \secref{sec:krn-proofs}.

\begin{prop}[Fibrewise restriction commutes with composition] \label{prp:fib-res-comp}
  $(k\circ h)[b] = k[b]\circ h[b]$, whenever $k\circ h$ exists.
\end{prop}

It is moreover easy to check that fibrewise restriction preserves identity kernels, which, together with the preceding proposition, means that it is functorial.
In particular, given $b:J\to B$ in $\QBS$, fibrewise restriction yields a reindexing functor $\Delta_b:\Ka_B\to\Ka_J$.

\begin{cor}[Substitution] \label{cor:sub-func}
  Suppose $b:J\to B$ in $\QBS$.
  Then we have a functor $\Delta_b:\Ka_B\to\Ka_J$ as follows.
  On objects, $\Delta_b$ acts to map $E\xto{p}B$ to $p[b]\xto{b^*p}J$.
  On morphisms, it maps $k$ to $k[b]$.
\end{cor}

\begin{prop}[The stochastic self-indexing, and resulting fibration] \label{prp:K-fib}
  Together with the substitution functors $\Delta$, the mapping $\Ka:B\mapsto\Ka_B$ defines a pseudo functor $\QBS\op\to\Cat$.
  Pseudo functoriality follows from the (pseudo) functoriality of pullback and the functoriality of composition.
  Applying the Grothendieck construction to this indexed category yields a fibration $\int\Ka\to\QBS$.
\end{prop}

It remains to exhibit left adjoints to the substitution functors $\Delta_b$.

\begin{prop}[Dependent sum functor] \label{prp:K-dep-sum-func}
  Suppose $f:C\to B$ is a morphism in $\QBS$.
  Then there is a functor $\Sigma_f:\Ka_C\to\Ka_B$ defined on objects $E\xto{p}C$ by post-composition, $\Sigma_f(p) := E\xto{p}C\xto{f}B$ and on kernels as the identity.
\end{prop}

Note that, fibrewise, we have $\Sigma_f(p)[b] \cong p[f^*b]$.
Since $f^*$ is pullback in $\QBS$ and $\Delta$ is defined by pullback, this is the first sign of the adjointness of $\Sigma$ and $\Delta$.
See \secref{sec:krn-proofs} for the proof of the proposition:

\begin{prop}[$\Sigma \dashv \Delta$] \label{prp:K-sig-del-adj}
  Suppose $f:B\to C$ in $\QBS$.
  Then $\Sigma_f$ is left adjoint to $\Delta_f$.
\end{prop}

Thus $\Ka$ yields a bifibration.
But this is not all we need: for copy-composition to be well-defined via pull-push, we need the functors $\Sigma$ and $\Delta$ to satisfy the Beck-Chevalley condition.
This result follows because the self-indexing of $\QBS$ satisfies Beck-Chevalley; but an explicit proof may be found in \secref{sec:krn-proofs}.

\begin{prop}[$\Ka$ satisfies Beck-Chevalley] \label{prp:K-b-c}
  Suppose we have a pullback square in $\QBS$ of the form
\[\begin{tikzcd}[sep=small]
	& P \\
	E && F \\
	& B
	\arrow["\pi"', from=1-2, to=2-1]
	\arrow["\rho", from=1-2, to=2-3]
	\arrow["p"', from=2-1, to=3-2]
	\arrow["q", from=2-3, to=3-2]
	\arrow["\lrcorner"{anchor=center, pos=0.125, rotate=-45}, draw=none, from=1-2, to=3-2]
\end{tikzcd} \; .\]
  Then there is a natural isomorphism $\Sigma_\rho\,\Delta_\pi \, \cong \, \Delta_q\,\Sigma_p$.
\end{prop}

\begin{rmk}
  In fact, the bifibration $\int\Ka$ has more structure than we have space to explore in the present study: each fibre has a tensor product that makes it into a monoidal fibration; additionally, if we extend to \textit{s-finite} kernels \parencite{Vakar2018SFinite}, each fibre has biproducts (and, with an inconsequential modification, a zero object); and thus the subcategory of probability kernels in each fibre is a \textit{commutative effectus} \parencite{Cho2015Introduction}.
  Moreover, these structures interact nicely with the base change functors.
  We leave the exposition of all this to future work.
\end{rmk}

\subsection{Decorating spans with sections}

Our task now is to define a categorical structure in which morphisms include graphs (in the functional sense) and compose by copy-composition.
That is, given the `graphs' of $f:X\to Y$ and $g:Y\to Z$ in some underlying monoidal category, we wish to be able to compose them to yield
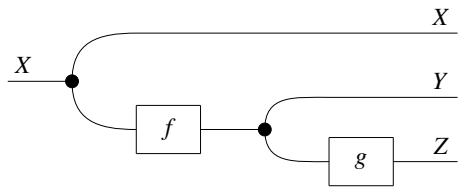
\[\scalebox{0.85}{\begin{tikzpicture}
	\begin{pgfonlayer}{nodelayer}
		\node [style=none] (0) at (-2, -1) {$f$};
		\node [style=none] (1) at (-3, -0.25) {};
		\node [style=none] (2) at (-1, -0.25) {};
		\node [style=none] (3) at (-3, -1.75) {};
		\node [style=none] (4) at (-1, -1.75) {};
		\node [style=none] (5) at (-3, -1) {};
		\node [style=none] (8) at (7, 2) {};
		\node [style=bcopy] (9) at (-5, 0.5) {};
		\node [style=none] (10) at (-3, 2) {};
		\node [style=none] (11) at (-7, 0.5) {};
		\node [style=none] (12) at (-6.5, 1) {$X$};
		\node [style=none] (13) at (6.5, 2.5) {$X$};
		\node [style=none] (15) at (4, -2) {$g$};
		\node [style=none] (16) at (3, -1.25) {};
		\node [style=none] (17) at (5, -1.25) {};
		\node [style=none] (18) at (3, -2.75) {};
		\node [style=none] (19) at (5, -2.75) {};
		\node [style=none] (20) at (3, -2) {};
		\node [style=none] (21) at (5, -2) {};
		\node [style=none] (22) at (7, -2) {};
		\node [style=none] (23) at (7, 0) {};
		\node [style=bcopy] (24) at (1, -1) {};
		\node [style=none] (25) at (3, 0) {};
		\node [style=none] (26) at (-1, -1) {};
		\node [style=none] (28) at (6.5, 0.5) {$Y$};
		\node [style=none] (29) at (6.5, -1.5) {$Z$};
	\end{pgfonlayer}
	\begin{pgfonlayer}{edgelayer}
		\draw (11.center) to (9);
		\draw (8.center)
			 to (10.center)
			 to [in=90, out=180, looseness=1.25] (9.center)
			 to [in=-180, out=-90, looseness=1.25] (5.center);
		\draw (1.center) to (2.center);
		\draw (3.center) to (4.center);
		\draw (1.center) to (3.center);
		\draw (2.center) to (4.center);
		\draw (26.center) to (24);
		\draw (23.center)
			 to (25.center)
			 to [in=90, out=180, looseness=1.25] (24.center)
			 to [in=-180, out=-90, looseness=1.25] (20.center);
		\draw (21.center) to (22.center);
		\draw (16.center) to (17.center);
		\draw (18.center) to (19.center);
		\draw (16.center) to (18.center);
		\draw (17.center) to (19.center);
	\end{pgfonlayer}
\end{tikzpicture}} \; .\]
This composite is a section of the composite projection $X\times Y\times Z\to X\times Y\to X$.
In turn, we may observe that this projection is the left leg of the composite span
\[\begin{tikzcd}[cramped,sep=scriptsize]
	&& {X\times Y\times Z} \\
	& {X\times Y} && {Y\times Z} \\
	X && Y && Z
	\arrow["{\pi^{XY}_X}"', from=2-2, to=3-1]
	\arrow["{\pi^{XY}_Y}", from=2-2, to=3-3]
	\arrow["{\pi^{YZ}_Y}"', from=2-4, to=3-3]
	\arrow[from=1-3, to=2-2]
	\arrow[from=1-3, to=2-4]
	\arrow["{\pi^{YZ}_Z}", from=2-4, to=3-5]
	\arrow["\lrcorner"{anchor=center, pos=0.125, rotate=-45}, draw=none, from=1-3, to=3-3]
\end{tikzcd}\]
and that the graphs of $f$ and $g$ are sections of the left legs of the constituent spans, respectively.
Thus, if we can interpret $\mathsf{graph}(f)$ and $\mathsf{graph}(g)$ bifibrationally, we could write their composite (depicted above) as 
$\Sigma_{\pi^{XY}_X}\,\Delta_{\pi^{XY}_Y}\,\bigl(\mathsf{graph}(g)\bigr)\circ\mathsf{graph}(f)$.
The operation $\Sigma_{\pi^{XY}_X}\,\Delta_{\pi^{XY}_Y}$ is known as a ``pull-push'' operation \parencite{Schreiber2014Quantization}, because one first pulls back along $\pi^{XY}_Y$ and then pushes forward along $\pi^{XY}_X$.%

When $f$ and $g$ are morphisms in a finitely complete category, we can immediately make sense of this pull-push composition in the codomain fibration.
Likewise, when they are s-finite kernels, we can interpret this story in $\Ka$, and this is the situation that allows us to construct Bayesian networks and recover the chain rule of the relative entropy.
Nonetheless, all we require is a bifibration satisfying Beck-Chevalley (and a section of it).
The idea is then to decorate spans in the base of the bifibration with sections of the left leg, and then to compose them by pullback on the spans and pull-push on the decorations.
Thus we have a ``decorated span'' construction, in the sense of \textcite{Fong2015Decorated} and \textcite{Patterson2023Structured}.

\textcite{Patterson2023Structured} observed that adding compositional decorations to categorical constructions often signals a Grothendieck fibration.
Since spans (and cospans) collect into double categories, the decorated (co)span construction can be reconceived as a double Grothendieck construction, giving an extra dimension of composition to that originally proposed by \textcite{Fong2015Decorated}, and newly allowing the decorations to depend on whole (co)spans (rather than merely their apices), a freedom we will need.
Following \textcite[Theorem 3.51]{Cruttwell2022Double}, the double Grothendieck construction produces a double fibration from an indexed double category, \textit{i.e.} a lax double pseudo functor whose codomain is the double 2-category $\SSpan(\Cat)$ of spans of categories; see \textcite[§3.1]{Patterson2023Structured} for a summary.

\begin{defn} \label{def:lax-dbl-func}
  A \textit{lax double pseudo functor} $F:\Bbb{C}\to\Bbb{D}$ consists of a pair of pseudo functors $F_0:\Bbb{C}_0\to\Bbb{D}_0$ (from the vertical category of $\Bbb{C}$ to that of $\Bbb{D}$) and $F_1:\Bbb{C}_1\to\Bbb{D}_1$ (from the horizontal category of $\Bbb{C}$ to that of $\Bbb{D}$), along with pseudo natural transformations $\mu$ and $\eta$ witnessing the compatibility of $F_1$ and $F_0$ with the double category structures of $\Bbb{C}$ and $\Bbb{D}$, as in the diagrams
\[
\begin{tikzcd}
	{\Bbb{C}_1\times_{\Bbb{C}_0}\Bbb{C}_1} && {\Bbb{C}_1} \\
	\\
	{\Bbb{D}_1\times_{\Bbb{D}_0}\Bbb{D}_1} && {\Bbb{D}_1}
	\arrow["{\diamond_{\Bbb{C}}}", from=1-1, to=1-3]
	\arrow["{F_1\times F_1}"', from=1-1, to=3-1]
	\arrow["{\diamond_{\Bbb{D}}}"', from=3-1, to=3-3]
	\arrow["{F_1}", from=1-3, to=3-3]
	\arrow["\mu", shorten <=19pt, shorten >=19pt, Rightarrow, from=3-1, to=1-3]
\end{tikzcd}
\qquad\qquad
\begin{tikzcd}
	{\Bbb{C}_0} && {\Bbb{C}_1} \\
	\\
	{\Bbb{D}_0} && {\Bbb{D}_1}
	\arrow["{I_{\Bbb{C}}}", from=1-1, to=1-3]
	\arrow["{I_{\Bbb{D}}}", from=3-1, to=3-3]
	\arrow["{F_0}"', from=1-1, to=3-1]
	\arrow["{F_1}", from=1-3, to=3-3]
	\arrow["\eta", shorten <=17pt, shorten >=17pt, Rightarrow, from=3-1, to=1-3]
\end{tikzcd}
\]
such that four standard coherence axioms \parencite[Def. 3.12]{Cruttwell2022Double} are satisfied.
($\diamond$ denotes the `external' (2-cell) composition of a double category structure and $I$ its unit, so that, in the context of indexed double categories, $\mu$ implements the \textit{indexed} external composition, and $\eta$ implements \textit{its} unit.)

An \textit{indexed double category} is a lax double pseudo functor with codomain $\SSpan(\Cat)$.
\end{defn}

Suppose now that $\Ea\xto{\pi}\Ba$ is a bifibration satisfying Beck-Chevalley, equipped with a section $\iota:\Ba\to\Ea$, where $\Ba$ has all pullbacks; we will continue to write $\Ea_B$ for the fibre over $B:\Ba$, and $\Sigma\dashv\Delta$ for the base change functors, and assume that $\Delta$ preserves $\iota$.
Our model for $\Ea$ will be either $\int\Ka$ or the codomain fibration of a finitely complete category; and for $\iota$ the mapping $B\mapsto\bigl(\id_B:B\to B\bigr)$, so that morphisms out of $\iota$ are understood as sections (and thus models for terms).
We will construct an indexed double category $S:\SSpan(\Ba)_{pb}\op\to\SSpan(\Cat)$, where $\SSpan(\Ba)_{pb}$ is the sub-double category of $\SSpan(\Ba)$ whose 2-cells are Cartesian morphisms of spans (a restriction necessary for morphisms of decorations to be well-defined), and the $\op$ is taken vertically.

The vertical category of $\SSpan(\Ba)_{pb}$ is $\Ba$, and the vertical 2-category of $\SSpan(\Cat)$ is $\Cat$; but the vertical decorations will be trivial: $S_0$ is simply the constant functor on the terminal category $\mathbf{1}$.

There is a simplification in the horizontal direction, too.
The horizontal category of $\SSpan(\Ba)_{pb}$ is $\Ba^{\{\bullet\leftarrow\bullet\to\bullet\}}_{pb}$, the category of diagrams in $\Ba$ of shape $\bullet\leftarrow\bullet\to\bullet$, with morphisms restricted to those natural transformations that are Cartesian (\textit{i.e.}, constituting pullback squares).
$S_1$ must map a span $m=\bigl(A\xleftarrow{a}E\xto{b}B\bigr)$ in $\Ba$ to a span $\mathbf{1}\leftarrow S_1(m)\to\mathbf{1}$ of categories.
Because $\mathbf{1}$ is terminal in $\Cat$, the legs of the latter span are accordingly determined.
There is a further simplification: $S_1(m)$ will be a discrete category (a set).
Indeed, we define $S_1(m)$ to be the set of sections of the left leg of $m$ in $\Ea$.
That is, $S_1(m) := \Ea_A(\iota A, \Sigma_a\,\iota E)$.

Given a morphism $f:m\to m'$ in $\SSpan(\Ba)_{pb}$ as in the diagram
\[\begin{tikzcd}[sep=scriptsize]
	A & E & B \\
	{A'} & {E'} & {B'}
	\arrow["a"', from=1-2, to=1-1]
	\arrow["{f_l}"', from=1-1, to=2-1]
	\arrow["{a'}"', from=2-2, to=2-1]
	\arrow[from=1-2, to=2-2]
	\arrow["\llcorner"{anchor=center, pos=0.125}, draw=none, from=1-2, to=2-1]
	\arrow["b", from=1-2, to=1-3]
	\arrow["{b'}", from=2-2, to=2-3]
	\arrow["{f_r}", from=1-3, to=2-3]
	\arrow["\lrcorner"{anchor=center, pos=0.125}, draw=none, from=1-2, to=2-3]
\end{tikzcd} \; ,\]
$S_1(f)$ must be a function $\Ea_{A'}(\iota A', \Sigma_{a'}\,\iota E')\to\Ea_A(\iota A, \Sigma_a\,\iota E)$, which we take to be given by $\Delta_{f_l}$.
This is well defined because $\Delta_{f_l}\,\iota A' = \iota_A$ \textit{ex hypothesi}, $\Delta_{f_l}\,\Sigma_{a'}\,\iota E' \cong \Sigma_a\,\Delta_f\,\iota E'$ by Beck-Chevalley, and $\Delta_f\,\iota E' = \iota E$ \textit{ex hypothesi}.
Moreover, since $\Delta_{f_l}$ is functorial, so is $S_1(f)$.

The copy-composition of sections is implemented by the transformation $\mu$, whose component at a pair of composable spans $m=\bigl(A\xleftarrow{a}E\xto{b}B\bigr)$ and $n=\bigl(B\xleftarrow{b'}F\xto{c}C\bigr)$ must be a function $\mu_{m,n}:S_1(m)\times S_1(n) \to S_1(n\diamond m)$, which we take to be defined by pull-push.
That is, $\mu_{m,n}(\sigma,\tau) := \Sigma_a\,\Delta_b(\tau)\circ\sigma$.
Associativity of $\mu$ follows by Beck-Chevalley; the explicit proof is in Appendix \ref{apdx:pull-push}.
The unit $\eta$ is defined over each identity span $E=E=E$ by the function $\eta_E:\mathbf{1}\to\Ea_E(\iota E,\iota E)$ determined by the identity on $\iota_E$.
Note that this makes pull-push copy-composition strictly unital, unlike in \textcite{Smithe2023Approximate}.

Applying the double Grothendieck construction \parencite[Theorem 3.51]{Cruttwell2022Double} to $S$ yields a double category $\Sbb$, doubly fibred over $\SSpan(\Ba)_{pb}$, which has the following structure:
\begin{enumerate}
\item The objects of $\Sbb$ are the objects of $\Ba$.
\item A vertical 1-cell is a morphism of $\Ba$, and vertical composition is as in $\Ba$.
\item A horizontal 1-cell $A\nrightarrow B$ is a quadruple $(E,a,b,\sigma)$ where $\bigl(A\xleftarrow{a}E\xto{b}B\bigr)$ is a span with $\sigma$ being a section of the left leg in $\Ea_A$; \textit{i.e.}, $\sigma \in \Ea_A(\iota A,\Sigma_a\,\iota E)$.
\item A 2 cell from $(E,a,b,\sigma):A\nrightarrow B$ to $(E',a',b',\sigma'):A'\nrightarrow B'$ is a Cartesian morphism of spans $(f_l,f,f_r):\bigl(A\xleftarrow{a}E\xto{b}B\bigr)\Rightarrow\bigl(A'\xleftarrow{a'}E'\xto{b'}B'\bigr)$ such that $\sigma = \Delta_{f_l}\,\sigma'$.%
\item Horizontal composition is by pullback on the spans (and their morphisms) and by pull-push on the sections, as described above.
\end{enumerate}

That is to say, we have the following theorem (given the details in Appendix \ref{apdx:pull-push}).

\begin{thm}[Copy-composition of sections] \label{thm:S}
  Given a bifibration $\Ea\to\Ba$ over a finitely complete base, satisfying Beck-Chevalley, and equipped with a section $\iota:\Ba\to\Ea$ preserved by pullback, the foregoing data define a pseudo double category $\Sbb$ doubly fibred over $\SSpan(\Ba)_{pb}$.
\end{thm}

By instantiating $\Sbb$ in $\int\Ka$, we can formally recover the factors of a Bayesian network described by \textcite{Fong2013Causal} and introduced at the beginning of this section.
There is a lax embedding of $\sfKrn$ horizontally into $\Sbb$ (so instantiated), which maps a kernel to its graph.
Let $\sigma$ denote the image of the kernel $X_i\mid\mathrm{pa}(i)$ under this embedding, which has the effect of precomposing it with copiers.
Then, to tensor with the identity wires $\otimes_{j\notin\mathrm{pa}(i)}\,X_j$, we can pull $\sigma$ back along the projection $\prod_{j<i}\,X_j \to \prod_{j\in\mathrm{pa}(i)}\,X_j$, as in the diagram
\[\begin{tikzcd}[cramped,row sep=small]
	{\prod_{j<i}\,X_j} & {\prod_{j\leq i}\,X_j} & {\prod_{j\leq i}\,X_j} \\
	\\
	{\prod_{j\in\mathrm{pa}(i)}\,X_j} & {\prod_{j\in\mathrm{pa}(i)}X_j\times X_i} & {X_i}
	\arrow[from=1-2, to=1-1]
	\arrow[from=3-2, to=3-1]
	\arrow[from=3-2, to=3-3]
	\arrow[Rightarrow, no head, from=1-2, to=1-3]
	\arrow[from=1-3, to=3-3]
	\arrow[from=1-1, to=3-1]
	\arrow[from=1-2, to=3-2]
	\arrow["\lrcorner"{anchor=center, pos=0.125, rotate=-90}, draw=none, from=1-2, to=3-1]
	\arrow["{\sigma_{\leq i}}", curve={height=-12pt}, dashed, from=1-1, to=1-2]
	\arrow["\sigma", curve={height=-12pt}, dashed, from=3-1, to=3-2]
\end{tikzcd} \; .\]
The upper decorated span is the horizontal morphism corresponding to Fong's factor.
Finally, note that whereas traditional statistical modelling is restricted to sections of product types (as above), using $\Ka$ and $\Sbb$ means we can build probabilistic models involving general dependent types.

\section{Undirected models: open factor graphs} \label{sec:fg}

\begin{figure}[h]
  \[\scalebox{0.8}{\begin{tikzpicture}
	\begin{pgfonlayer}{nodelayer}
		\node [style=box-7mm] (0) at (-6, 2) {$f_1$};
		\node [style=box-7mm] (1) at (-1, 2) {$f_2$};
		\node [style=box-7mm] (2) at (5, 2) {$f_3$};
		\node [style=box-7mm] (3) at (2, -1) {$f_5$};
		\node [style=dot-5mm] (4) at (2, 2) {$D$};
		\node [style=dot-5mm] (5) at (-3.5, 2) {$C$};
		\node [style=box-7mm] (9) at (-3.5, 7) {$f_0$};
		\node [style=box-7mm] (10) at (11, 2) {$f_4$};
		\node [style=dot-5mm] (11) at (-4.75, 4.5) {$A$};
		\node [style=dot-5mm] (12) at (8, 2) {$E$};
		\node [style=dot-5mm] (13) at (-2.25, 4.5) {$B$};
	\end{pgfonlayer}
	\begin{pgfonlayer}{edgelayer}
		\draw (1) to (4);
		\draw (4) to (2);
		\draw (3) to (4);
		\draw (9) to (11);
		\draw (11) to (0);
		\draw (9) to (13);
		\draw (13) to (1);
		\draw (0) to (5);
		\draw (5) to (1);
		\draw (2) to (12);
		\draw (12) to (10);
	\end{pgfonlayer}
\end{tikzpicture}}\]
  \caption{A factor graph, in the style of \textcite[§2.1.3]{Wainwright2007Graphical}.}
  \label{fig:fg-1}
\end{figure}

The preceding section explained how copy-composition is central to Bayesian networks: \textit{directed} probabilistic graphical models.
In this section, we shall see that it is also central to the construction of undirected models known as \textit{factor graphs} \parencite{Dauwels2007Variational,Wainwright2007Graphical}.
A factor graph is a graphical representation of the factorization of a function, which is often but not necessarily interpreted as the density of a probability distribution.
Thus Figure~\ref{fig:fg-1} represents a function $f$ of the form
$$f(a,b,c,d,e) = f_0(a,b)\, f_1(a,c)\, f_2(b,c,d)\, f_3(d,e)\, f_4(e)\, f_5(d)$$
where the variables have the types $a:A,b:B,c:C,d:D$.
Note that we could see $f$ as composed of `factors' $f_i$; it is this composition that we explore here.

Typically, the factors of a factor graph are functions into some object of scalars; for example, we might interpret $f_2$ (above) to have the type $B\times C\times D \to \Rb_+$, for some spaces $B,C,D$.
Should these spaces be $\mathbb{R}_+$-modules, then we might suppose $f_2$ to be a linear functional or \textit{covector} on their product.
$\Rb_+$-modules form a category $\Mod_{\Rb_+}$ which is equipped with a tensor product $\otimes$, whose unit is the 1-dimensional space of scalars $\Rb_+$, often denoted $I$.
Covectors on $X$ then correspond to \textit{costates}, morphisms $X\to I$.
This suggests that we may internalize the notion of factor in a monoidal category $(\Ca,\otimes,I)$ simply as the costates in $\Ca$.

This first step towards an abstract treatment should be taken with care, as any semicartesian monoidal category, such as $\Set$ or any Markov category (in the sense of \textcite{Fritz2019synthetic}), will only have trivial costates, the monoidal unit there being by definition terminal.
The composition of factors then proceeds by copy-composition; notice for example the duplication of the variable $b:B$ in the expression for $f$ above, or in the following string-diagrammatic representation of its $f_0,f_2$ factor:
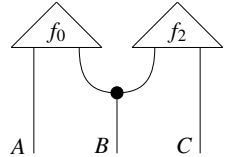
\[\scalebox{0.8}{\begin{tikzpicture}
	\begin{pgfonlayer}{nodelayer}
		\node [style=none] (0) at (-3.5, 1) {};
		\node [style=none] (1) at (-2, 2.5) {};
		\node [style=none] (2) at (-0.5, 1) {};
		\node [style=none] (3) at (-2, 1.5) {$f_0$};
		\node [style=none] (4) at (0.5, 1) {};
		\node [style=none] (5) at (2, 2.5) {};
		\node [style=none] (6) at (3.5, 1) {};
		\node [style=none] (7) at (2, 1.5) {$f_2$};
		\node [style=none] (8) at (-1.25, 1) {};
		\node [style=none] (9) at (1.25, 1) {};
		\node [style=bcopy] (10) at (0, -0.5) {};
		\node [style=none] (11) at (0, -2.5) {};
		\node [style=none] (13) at (-2.75, 1) {};
		\node [style=none] (14) at (-2.75, -2.5) {};
		\node [style=none] (17) at (2.75, 1) {};
		\node [style=none] (18) at (2.75, -2.5) {};
		\node [style=none] (19) at (-3.25, -2.25) {$A$};
		\node [style=none] (21) at (-0.5, -2.25) {$B$};
		\node [style=none] (22) at (2.25, -2.25) {$C$};
	\end{pgfonlayer}
	\begin{pgfonlayer}{edgelayer}
		\draw (0.center) to (1.center);
		\draw (1.center) to (2.center);
		\draw (0.center) to (2.center);
		\draw (4.center) to (5.center);
		\draw (5.center) to (6.center);
		\draw (4.center) to (6.center);
		\draw (8.center)
			 to [in=-180, out=-90, looseness=1.25] (10.center)
			 to [in=-90, out=0, looseness=1.25] (9.center);
		\draw (11.center) to (10);
		\draw (13.center) to (14.center);
		\draw (17.center) to (18.center);
	\end{pgfonlayer}
\end{tikzpicture}}\]
This suggests that an appropriate categorical setting for factor graphs is that of \textit{copy-discard categories}\footnote{
A \textit{copy-discard category} is a monoidal category in which every object is equipped with a comonoid structure (with respect to the ambient monoidal structure) \parencite{Cho2017Disintegration}.
A \textit{Markov category} is a semicartesian copy-discard category \parencite{Fritz2019synthetic}.
Precisely because of the existence of non-trivial costates, copy-discard categories often end up being more useful in modelling practice than Markov categories, despite the utility of the latter for synthetic probability theory.}, of which $\sfKrn$ (or any fibre of $\Ka$) is an example.

Because factor graphs are a form of undirected graphical model, it is natural to expect them to fit into the framework of undirected wiring diagrams
\parencite{Spivak2013Operad} \parencite[Chapter 7]{Yau2018Operads}.
Undirected wiring diagrams are algebras for an operad of cospans \parencite{Fong2019Hypergraph}, and one way to obtain such algebras is through the decorated cospans construction, originally due to \textcite{Fong2015Decorated}.
(But note that the decorated spans of the previous section are not undirected in the same way, because the decorations there are asymmetric.)
The main result of this section is to show that `open' factor graphs may be obtained and composed precisely through the double-categorical generalization of this construction due to \textcite{Patterson2023Structured}.

\begin{rmk}
  \textcite{Fong2019Hypergraph} show that cospan-algebras are equivalent to hypergraph categories, in which each morphism has the form of a hypergraph.
  Consequently, this is also the case for factor graphs.%
\end{rmk}

The remainder of this section is dedicated to constructing a double category indexed by cospans, \textit{i.e.} a double pseudo functor $F$, and instantiating its double Grothendieck construction, the double category $\FG$.
The domain of $F$ is the double category of cospans of finite sets, whose vertical category $\CCospan(\FinSet)_0$ is simply $\FinSet$, and whose horizontal category $\CCospan(\FinSet)_1$ is the category of diagrams in $\FinSet$ of shape $\bullet\to\bullet\leftarrow\bullet$, \textit{i.e.} the functor category $\FinSet^{\{\bullet\to\bullet\leftarrow\bullet\}}$.
The codomain of $F$ is $\SSpan(\Cat)$, whose vertical 2-category is $\Cat$, and whose horizontal 2-category is $\Cat^{\{\bullet\leftarrow\bullet\to\bullet\}}$.

Henceforth, we will assume an ambient copy-discard category $(\Ca,I,\otimes)$.
Because finite sets are not ordered and we plan to define factors over accordingly indexed monoidal products, we must additionally assume that the monoidal structure $(I,\otimes)$ on $\Ca$ is symmetric.

There will be fewer simplifications involved in the definition of $F$ (versus $S$), and so we split the construction into four subsections.

\subsection{Vertical decoration: factors' interfaces}

The pseudo functor $F_0:\FinSet\to\Cat$ assigns categories of decorations to finite sets.
For us, these decorations will be the exposed interfaces of factors, given by a discrete diagram in $\Ca$: a functor $\disc X\to\Ca$, for some finite set $X$, where $\disc X$ is the discrete category on $X$.
However, were we just to define $F_0X$ to be the functor category $\Cat(\disc X, \Ca)$, we would run into problems later with the composition of 2-cells.
This is because vertical morphisms --- decorated by morphisms in the image of $F_0$ --- may be used to transform the interface types, but the composition of factors proceeds by copying.
For this to be compatible with morphisms of interfaces, we must restrict the latter to be comonoid homomorphisms, which (in probabilistic contexts) we can often understand as deterministic functions.

Therefore, on objects $X:\FinSet$, let $F_0X$ be the functor category $\Cat(\disc X, \Comon(\Ca,I,\otimes))$ where $\Comon(\Ca,I,\otimes)$ denotes the (wide) monoidal subcategory of $\Ca$ whose morphisms are the comonoid homomorphisms.
Note that, because the objects $\chi$ of $F_0X$ are functors on a discrete category, a morphism $\varphi:\chi\to\chi'$ (hence a natural transformation) is simply given by an $X$-indexed family of comonoid homomorphisms $\varphi_x:\chi(x)\to\chi'(x)$ in $\Ca$.

Given a morphism $f:X\to Y$ in $\FinSet$, we let $F_0$ return a functor $F_0X\to F_0Y$ which we denote by $f_*$ and define as follows.
If $\chi$ is an object of $F_0X$ then we define $f_*\chi$ by the mapping $y\mapsto\otimes_{x:f^{-1}(y)}\chi(x)$ (for $y:Y$).
Then, if $\varphi:\chi\to\chi'$ is a morphism of $F_0X$, we define each $y$-component $(f_*\varphi)_y$ of $f_*\varphi$ to be $\otimes_{x:f^{-1}(y)}\varphi_x$.
Note that this definition yields components of the correct type
$$(f_*\varphi)_y : (f_*\chi)(y) = \otimes_{x:f^{-1}(y)}\chi(x) \xrightarrow{\otimes_{x:f^{-1}(y)}\varphi_x} \otimes_{x:f^{-1}(y)}\chi'(x) = (f_*\chi')(y) \; .$$
Naturality is trivially satisfied, because $\chi$ and $\chi'$ are discrete functors, and the functoriality of the reindexing $f_*$ follows from the functoriality of $\otimes$.
Likewise, the pseudofunctoriality of $F_0$ itself follows from the functoriality of inverse images and the functoriality and symmetry of $\otimes$, so that $\otimes_{y:g^{-1}(z)} \otimes_{x:f^{-1}(y)} \cong \otimes_{x:(g\circ f)^{-1}(z)}$.

\subsection{Horizontal decoration: `open' factors}

The pseudo functor $F_1 : \FinSet^{\{\bullet\to\bullet\leftarrow\bullet\}} \to \Cat^{\{\bullet\leftarrow\bullet\to\bullet\}}$ assigns spans of categories to cospans of finite sets, compatibly with $F_0$.
We interpret a cospan $A\xrightarrow{a} X\xleftarrow{b} B$ as follows.
The apex set $X$ will be decorated with two kinds of data: first, a category of interfaces suited to factors in $\Ca$; and second, for each suitable interface, a category of factors on that interface (compatibly with morphisms of interfaces).
The legs of the cospan $a$ and $b$ will be used to expose those parts of the interface which are open to composition, and in this way the objects $A$ and $B$ will be decorated by $F_0A$ and $F_0B$ respectively; this is one of the indexed-double-categorical coherence axioms (`well-definition').

Thus, given a cospan $m=\bigl(A\xrightarrow{a} X\xleftarrow{b} B\bigr)$, we let $F_1$ return a span of categories $F_0A \xleftarrow{\pi_a} \int\underline{F_1}m \xrightarrow{\pi_b} F_0B$.
The apex category $\int\underline{F_1}m$ has for objects pairs $(\chi,f)$ where $\chi:\disc X\to\Ca$ is an interface and $f:\chi^\otimes\to I$ is a factor on that interface.
$\chi^\otimes$ denotes the monoidal product $\otimes_{x:X}\,\chi(x)$ in $\Ca$.
The morphisms of $\int\underline{F_1}m$ are factorizations of factors through transformations of interfaces; that is to say, a morphism $(\chi,f)\to(\chi',f')$ in $\int\underline{F_1}m$ is a morphism of interfaces $\varphi:\chi\to\chi'$ such that $f = \varphi^* f'$, where $\varphi^*$ denotes precomposition by $\varphi^\otimes := \otimes_{x:X}\,\varphi_x$.
The functors $\pi_a$ and $\pi_b$ forget the factors and project out those parts of the interfaces over $X$ that are exposed by $a$ and $b$ respectively.
($\int\underline{F_1}m$ and the functors $\pi_a,\pi_b$ are in fact obtained through universal constructions --- a limit and a Grothendieck construction --- as we explain in Appendix~\ref{sec:F1}.
Universality then yields the functorial action of $F_1$ on morphisms of cospans.)

\subsection{Factor composition} \label{sec:copy-comp}

The functors $F_0$ and $F_1$ describe how to attach factors to interfaces, and how these interfaces (and the factors on them) may be transformed by morphisms in $\Ca$.
But they do not tell us about the operation at the heart of the open factor graph construction: the horizontal (copy-)composition of factors, or the horizontal composition of their morphisms.
These operations are formalized by the external composition structure $(\mu,\eta)$ with which we now equip $F$.

Let us begin with $\mu$, a pseudo natural transformation whose component $\mu_{m,n}$ at a pair of composable cospans $m=\bigl(A\xrightarrow{a}X\xleftarrow{b}B\bigr)$ and $n=\bigl(B\xrightarrow{b'}Y\xleftarrow{c}C\bigr)$ must be a functor $\int\underline{F_1}m\times_{F_0B}\int\underline{F_1}n\to\int\underline{F_1}(n\diamond m)$.
Given factors $(\chi,f),(\gamma,g)$ in its domain, $\mu_{m,n}$ must return an interface over $X+_BY$ and a factor on that interface.
The former is given by the copairing $[\chi,\gamma]_B$ of $\chi$ and $\gamma$ over $B$, while the latter is given by copy-composition.
For each $j:X+_BY$, copying induces a morphism
$$\delta_{\chi,\gamma}^j:[\chi,\gamma]_B(j) \to \bigotimes_{i:[\iota^X_B,\iota^Y_B]^{-1}(j)}\, [\chi,\gamma](i)$$
in $\Ca$ which is unique up to coassociativity.
We then define $\delta_{\chi,\gamma} := \otimes_{j:X+_BY}\;\delta_{\chi,\gamma}^j$, which has the type $[\chi,\gamma]_B^\otimes \to [\chi,\gamma]^\otimes$.
We can compose $f\otimes g$ after this higher-arity copier to yield a factor on the interface $[\chi,\gamma]_B$.
Thus $\mu_{m,n}\bigl((\chi,f),(\gamma,g)\bigr) := \bigl([\chi,\gamma]_B,(f\otimes g)\circ\delta_{\chi,\gamma}\bigr)$.
(For a more detailed construction and proof that this yields a pseudo natural transformation, see Appendix~\ref{sec:F-mu}.)

To see $\mu$ in action, suppose $A$ has $l$-many elements, $C$ has $n$-many elements, and the coupled interface $B$ has $m$-many elements, with $B_1$ shared once between $f$ and $g$ but $B_m$ shared an arbitrary number of times.
Then $\mu\bigl((\chi,f),(\gamma,g)\bigr)$ yields the factor
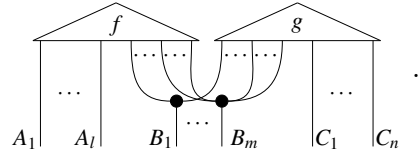
\[\scalebox{0.8}{\begin{tikzpicture}
	\begin{pgfonlayer}{nodelayer}
		\node [style=none] (0) at (-6.25, 1) {};
		\node [style=none] (1) at (-3.5, 2.25) {};
		\node [style=none] (2) at (-0.75, 1) {};
		\node [style=none] (3) at (-3.5, 1.5) {$f$};
		\node [style=none] (4) at (-0.25, 1) {};
		\node [style=none] (5) at (2.5, 2.25) {};
		\node [style=none] (6) at (5.25, 1) {};
		\node [style=none] (7) at (2.5, 1.5) {$g$};
		\node [style=none] (8) at (-2, 1) {};
		\node [style=none] (9) at (2, 1) {};
		\node [style=bcopy] (10) at (0, -1) {};
		\node [style=none] (11) at (0, -2.5) {};
		\node [style=none] (13) at (-6, 1) {};
		\node [style=none] (14) at (-6, -2.5) {};
		\node [style=none] (19) at (-6.5, -2.25) {$A_1$};
		\node [style=none] (21) at (0.75, -2.25) {$B_m$};
		\node [style=none] (23) at (-4, 1) {};
		\node [style=none] (24) at (-4, -2.5) {};
		\node [style=none] (25) at (-4.5, -2.25) {$A_l$};
		\node [style=none] (26) at (-5, -0.75) {$\cdots$};
		\node [style=none] (27) at (-3, 1) {};
		\node [style=none] (28) at (0, 1) {};
		\node [style=bcopy] (29) at (-1.5, -1) {};
		\node [style=none] (30) at (-1.5, -2.5) {};
		\node [style=none] (31) at (-2, -2.25) {$B_1$};
		\node [style=none] (32) at (-1, 1) {};
		\node [style=none] (33) at (1, 1) {};
		\node [style=none] (34) at (-1.5, 0.5) {$\cdots$};
		\node [style=none] (35) at (1.5, 0.5) {$\cdots$};
		\node [style=none] (36) at (3, 1) {};
		\node [style=none] (37) at (3, -2.5) {};
		\node [style=none] (38) at (3.5, -2.25) {$C_1$};
		\node [style=none] (39) at (5, 1) {};
		\node [style=none] (40) at (5, -2.5) {};
		\node [style=none] (41) at (5.5, -2.25) {$C_n$};
		\node [style=none] (42) at (4, -0.75) {$\cdots$};
		\node [style=none] (43) at (-2.5, 0.5) {$\cdots$};
		\node [style=none] (45) at (0.5, 0.5) {$\cdots$};
		\node [style=none] (46) at (-0.75, -1.75) {$\cdots$};
	\end{pgfonlayer}
	\begin{pgfonlayer}{edgelayer}
		\draw (0.center) to (1.center);
		\draw (1.center) to (2.center);
		\draw (0.center) to (2.center);
		\draw (4.center) to (5.center);
		\draw (5.center) to (6.center);
		\draw (4.center) to (6.center);
		\draw [in=-180, out=-90, looseness=1.25] (8.center) to (10);
		\draw [in=-90, out=0, looseness=1.25] (10) to (9.center);
		\draw (11.center) to (10);
		\draw (13.center) to (14.center);
		\draw (23.center) to (24.center);
		\draw [in=-180, out=-90, looseness=1.25] (27.center) to (29);
		\draw [in=-90, out=0, looseness=1.25] (29) to (28.center);
		\draw (30.center) to (29);
		\draw [in=-180, out=-90, looseness=1.25] (32.center) to (10);
		\draw [in=-90, out=0, looseness=1.25] (10) to (33.center);
		\draw (36.center) to (37.center);
		\draw (39.center) to (40.center);
	\end{pgfonlayer}
\end{tikzpicture}} \; .\]

The naturality of $\mu_{m,n}$ means that if we were to have $f'$ and $g'$ defined by transforming the interface $B$ along $\beta$ like so
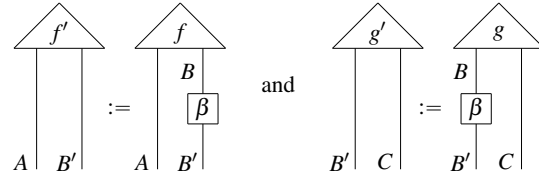
\[\scalebox{0.8}{
  \begin{tikzpicture}
	\begin{pgfonlayer}{nodelayer}
		\node [style=none] (0) at (0.5, 1.25) {};
		\node [style=none] (1) at (2, 2.75) {};
		\node [style=none] (2) at (3.5, 1.25) {};
		\node [style=none] (3) at (2, 1.75) {$f$};
		\node [style=none] (4) at (2.75, 1.25) {};
		\node [style=none] (6) at (2.75, -2.75) {};
		\node [style=none] (7) at (1.25, 1.25) {};
		\node [style=none] (8) at (1.25, -2.75) {};
		\node [style=none] (9) at (0.75, -2.5) {$A$};
		\node [style=none] (10) at (2.25, -2.5) {$B'$};
		\node [style=none] (11) at (2.75, 1.25) {};
		\node [style=none] (12) at (2.25, -0.25) {};
		\node [style=none] (13) at (3.25, -0.25) {};
		\node [style=none] (14) at (2.25, -1.25) {};
		\node [style=none] (15) at (3.25, -1.25) {};
		\node [style=none] (16) at (2.75, -0.75) {$\beta$};
		\node [style=none] (17) at (2.75, -0.25) {};
		\node [style=none] (18) at (2.75, -1.25) {};
		\node [style=none] (19) at (2.25, 0.5) {$B$};
		\node [style=none] (20) at (-3.5, 1.25) {};
		\node [style=none] (21) at (-2, 2.75) {};
		\node [style=none] (22) at (-0.5, 1.25) {};
		\node [style=none] (23) at (-2, 1.75) {$f'$};
		\node [style=none] (24) at (-1.25, 1.25) {};
		\node [style=none] (25) at (-2.75, 1.25) {};
		\node [style=none] (26) at (-1.25, 1.25) {};
		\node [style=none] (27) at (-1.25, -2.75) {};
		\node [style=none] (28) at (-2.75, -2.75) {};
		\node [style=none] (29) at (-3.25, -2.5) {$A$};
		\node [style=none] (30) at (-1.75, -2.5) {$B'$};
		\node [style=none] (31) at (0, -0.75) {$:=$};
	\end{pgfonlayer}
	\begin{pgfonlayer}{edgelayer}
		\draw (0.center) to (1.center);
		\draw (1.center) to (2.center);
		\draw (0.center) to (2.center);
		\draw (7.center) to (8.center);
		\draw (12.center) to (13.center);
		\draw (12.center) to (14.center);
		\draw (14.center) to (15.center);
		\draw (15.center) to (13.center);
		\draw (11.center) to (17.center);
		\draw (18.center) to (6.center);
		\draw (20.center) to (21.center);
		\draw (21.center) to (22.center);
		\draw (20.center) to (22.center);
		\draw (26.center) to (27.center);
		\draw (25.center) to (28.center);
	\end{pgfonlayer}
\end{tikzpicture}
  \quad and \quad
  \begin{tikzpicture}
	\begin{pgfonlayer}{nodelayer}
		\node [style=none] (0) at (0.75, 1.25) {};
		\node [style=none] (1) at (2.25, 2.75) {};
		\node [style=none] (2) at (3.75, 1.25) {};
		\node [style=none] (3) at (2.25, 1.75) {$g$};
		\node [style=none] (6) at (1.5, -2.75) {};
		\node [style=none] (7) at (3, 1.25) {};
		\node [style=none] (8) at (3, -2.75) {};
		\node [style=none] (9) at (2.5, -2.5) {$C$};
		\node [style=none] (10) at (1, -2.5) {$B'$};
		\node [style=none] (11) at (1.5, 1.25) {};
		\node [style=none] (12) at (1, -0.25) {};
		\node [style=none] (13) at (2, -0.25) {};
		\node [style=none] (14) at (1, -1.25) {};
		\node [style=none] (15) at (2, -1.25) {};
		\node [style=none] (16) at (1.5, -0.75) {$\beta$};
		\node [style=none] (17) at (1.5, -0.25) {};
		\node [style=none] (18) at (1.5, -1.25) {};
		\node [style=none] (19) at (1, 0.5) {$B$};
		\node [style=none] (20) at (-3.25, 1.25) {};
		\node [style=none] (21) at (-1.75, 2.75) {};
		\node [style=none] (22) at (-0.25, 1.25) {};
		\node [style=none] (23) at (-1.75, 1.75) {$g'$};
		\node [style=none] (25) at (-1, 1.25) {};
		\node [style=none] (26) at (-2.5, 1.25) {};
		\node [style=none] (27) at (-2.5, -2.75) {};
		\node [style=none] (28) at (-1, -2.75) {};
		\node [style=none] (29) at (-1.5, -2.5) {$C$};
		\node [style=none] (30) at (-3, -2.5) {$B'$};
		\node [style=none] (31) at (0, -0.75) {$:=$};
	\end{pgfonlayer}
	\begin{pgfonlayer}{edgelayer}
		\draw (0.center) to (1.center);
		\draw (1.center) to (2.center);
		\draw (0.center) to (2.center);
		\draw (7.center) to (8.center);
		\draw (12.center) to (13.center);
		\draw (12.center) to (14.center);
		\draw (14.center) to (15.center);
		\draw (15.center) to (13.center);
		\draw (11.center) to (17.center);
		\draw (18.center) to (6.center);
		\draw (20.center) to (21.center);
		\draw (21.center) to (22.center);
		\draw (20.center) to (22.center);
		\draw (26.center) to (27.center);
		\draw (25.center) to (28.center);
	\end{pgfonlayer}
\end{tikzpicture}
}\]
then copy-composing $f'$ and $g'$ along $B'$ yields the same as copy-composing $f$ and $g$ along $B$ and then transforming the $B$ interface of the result accordingly:
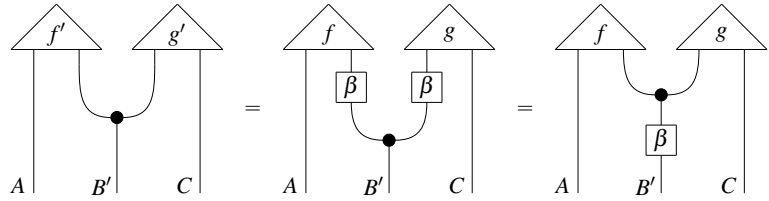
\[\scalebox{0.8}{\begin{tikzpicture}
	\begin{pgfonlayer}{nodelayer}
		\node [style=none] (0) at (-12.5, 1.75) {};
		\node [style=none] (1) at (-11, 3.25) {};
		\node [style=none] (2) at (-9.5, 1.75) {};
		\node [style=none] (3) at (-11, 2.25) {$f'$};
		\node [style=none] (4) at (-10.25, 1.75) {};
		\node [style=none] (5) at (-11.75, 1.75) {};
		\node [style=none] (6) at (-10.25, 1.75) {};
		\node [style=none] (7) at (-10.25, 1.25) {};
		\node [style=none] (8) at (-11.75, -3) {};
		\node [style=none] (9) at (-12.25, -2.75) {$A$};
		\node [style=none] (11) at (-8.5, 1.75) {};
		\node [style=none] (12) at (-7, 3.25) {};
		\node [style=none] (13) at (-5.5, 1.75) {};
		\node [style=none] (14) at (-7, 2.25) {$g'$};
		\node [style=none] (15) at (-6.25, 1.75) {};
		\node [style=none] (16) at (-7.75, 1.75) {};
		\node [style=none] (17) at (-7.75, 1.25) {};
		\node [style=none] (18) at (-6.25, -3) {};
		\node [style=none] (19) at (-6.75, -2.75) {$C$};
		\node [style=none] (20) at (-9.5, -2.75) {$B'$};
		\node [style=none] (21) at (-4.5, -0.25) {$=$};
		\node [style=none] (22) at (-3.5, 1.75) {};
		\node [style=none] (23) at (-2, 3.25) {};
		\node [style=none] (24) at (-0.5, 1.75) {};
		\node [style=none] (25) at (-2, 2.25) {$f$};
		\node [style=none] (26) at (-1.25, 1.75) {};
		\node [style=none] (28) at (-2.75, 1.75) {};
		\node [style=none] (29) at (-2.75, -3) {};
		\node [style=none] (30) at (-3.25, -2.75) {$A$};
		\node [style=none] (32) at (-1.25, 1.75) {};
		\node [style=none] (33) at (-1.75, 1) {};
		\node [style=none] (34) at (-0.75, 1) {};
		\node [style=none] (35) at (-1.75, 0) {};
		\node [style=none] (36) at (-0.75, 0) {};
		\node [style=none] (37) at (-1.25, 0.5) {$\beta$};
		\node [style=none] (38) at (-1.25, 1) {};
		\node [style=none] (39) at (-1.25, 0) {};
		\node [style=none] (41) at (0.5, 1.75) {};
		\node [style=none] (42) at (2, 3.25) {};
		\node [style=none] (43) at (3.5, 1.75) {};
		\node [style=none] (44) at (2, 2.25) {$g$};
		\node [style=none] (46) at (2.75, 1.75) {};
		\node [style=none] (47) at (2.75, -3) {};
		\node [style=none] (48) at (2.25, -2.75) {$C$};
		\node [style=none] (49) at (-0.5, -2.75) {$B'$};
		\node [style=none] (50) at (1.25, 1.75) {};
		\node [style=none] (51) at (0.75, 1) {};
		\node [style=none] (52) at (1.75, 1) {};
		\node [style=none] (53) at (0.75, 0) {};
		\node [style=none] (54) at (1.75, 0) {};
		\node [style=none] (55) at (1.25, 0.5) {$\beta$};
		\node [style=none] (56) at (1.25, 1) {};
		\node [style=none] (57) at (1.25, 0) {};
		\node [style=bcopy] (58) at (-9, -0.5) {};
		\node [style=none] (59) at (-9, -3) {};
		\node [style=bcopy] (60) at (0, -1.25) {};
		\node [style=none] (61) at (0, -3) {};
		\node [style=none] (62) at (5.5, 1.75) {};
		\node [style=none] (63) at (7, 3.25) {};
		\node [style=none] (64) at (8.5, 1.75) {};
		\node [style=none] (65) at (7, 2.25) {$f$};
		\node [style=none] (67) at (6.25, 1.75) {};
		\node [style=none] (69) at (7.75, 1.75) {};
		\node [style=none] (70) at (6.25, -3) {};
		\node [style=none] (71) at (5.75, -2.75) {$A$};
		\node [style=none] (72) at (9.5, 1.75) {};
		\node [style=none] (73) at (11, 3.25) {};
		\node [style=none] (74) at (12.5, 1.75) {};
		\node [style=none] (75) at (11, 2.25) {$g$};
		\node [style=none] (76) at (11.75, 1.75) {};
		\node [style=none] (78) at (10.25, 1.75) {};
		\node [style=none] (79) at (11.75, -3) {};
		\node [style=none] (80) at (11.25, -2.75) {$C$};
		\node [style=none] (81) at (8.5, -2.75) {$B'$};
		\node [style=bcopy] (82) at (9, 0.25) {};
		\node [style=none] (84) at (8.5, -0.75) {};
		\node [style=none] (85) at (9.5, -0.75) {};
		\node [style=none] (86) at (8.5, -1.75) {};
		\node [style=none] (87) at (9.5, -1.75) {};
		\node [style=none] (88) at (9, -1.25) {$\beta$};
		\node [style=none] (89) at (9, -0.75) {};
		\node [style=none] (90) at (9, -1.75) {};
		\node [style=none] (91) at (9, -3) {};
		\node [style=none] (92) at (4.5, -0.25) {$=$};
	\end{pgfonlayer}
	\begin{pgfonlayer}{edgelayer}
		\draw (0.center) to (1.center);
		\draw (1.center) to (2.center);
		\draw (0.center) to (2.center);
		\draw (6.center) to (7.center);
		\draw (5.center) to (8.center);
		\draw (11.center) to (12.center);
		\draw (12.center) to (13.center);
		\draw (11.center) to (13.center);
		\draw (16.center) to (17.center);
		\draw (15.center) to (18.center);
		\draw (22.center) to (23.center);
		\draw (23.center) to (24.center);
		\draw (22.center) to (24.center);
		\draw (28.center) to (29.center);
		\draw (33.center) to (34.center);
		\draw (33.center) to (35.center);
		\draw (35.center) to (36.center);
		\draw (36.center) to (34.center);
		\draw (32.center) to (38.center);
		\draw (41.center) to (42.center);
		\draw (42.center) to (43.center);
		\draw (41.center) to (43.center);
		\draw (46.center) to (47.center);
		\draw (51.center) to (52.center);
		\draw (51.center) to (53.center);
		\draw (53.center) to (54.center);
		\draw (54.center) to (52.center);
		\draw (50.center) to (56.center);
		\draw [in=-180, out=-90, looseness=1.25] (7.center) to (58);
		\draw [in=0, out=-90, looseness=1.25] (17.center) to (58);
		\draw (58) to (59.center);
		\draw [in=-180, out=-90, looseness=1.25] (39.center) to (60);
		\draw [in=-90, out=0, looseness=1.25] (60) to (57.center);
		\draw (60) to (61.center);
		\draw (62.center) to (63.center);
		\draw (63.center) to (64.center);
		\draw (62.center) to (64.center);
		\draw (67.center) to (70.center);
		\draw (72.center) to (73.center);
		\draw (73.center) to (74.center);
		\draw (72.center) to (74.center);
		\draw (76.center) to (79.center);
		\draw [in=-180, out=-90, looseness=1.25] (69.center) to (82);
		\draw [in=0, out=-90, looseness=1.25] (78.center) to (82);
		\draw (84.center) to (85.center);
		\draw (84.center) to (86.center);
		\draw (86.center) to (87.center);
		\draw (87.center) to (85.center);
		\draw (82) to (89.center);
		\draw (91.center) to (90.center);
	\end{pgfonlayer}
\end{tikzpicture}}\]
However, this only works for `determinsitic' transformations $\beta$: although the rules of open factor graphs do allow us to transform factors using arbitrary morphisms, we can only pull these transformations `outside' the composites when the transformations are comonoid homomorphisms.

It remains to define the unit $\eta$, a pseudo natural transformation whose component $\eta_X$ at each finite set $X$ is given by a functor $\eta_X : F_0X \to \int\underline{F_1}(\id_X)$.
We define $\eta_X$ using the counits of the comonoid structures on the interfaces on $X$.
That is to say, $\eta_X$ maps an interface $\chi:\disc X \to \Comon(\Ca)$ to the pair $(\chi, \epsilon_\chi)$, where here $\epsilon_\chi$ denotes the canonical discarding map $\chi^\otimes \to I$.
The functorial action of $\eta_X$ is obtained from the fact that all morphisms in $F_0X$ are comonoid homomorphisms, so if $\varphi$ is a morphism $\chi \to \chi'$ then $\epsilon_\chi = \varphi^*\epsilon_{\chi'}$.
Pseudo naturality of $\eta$ obtains likewise, and its unitality with respect to $\mu$ follows from the counitality of the comonoid counits.

\subsection{Putting it all together}

We have defined a double pseudo functor $F:\CCospan(\FinSet)\to\SSpan(\Cat)$, from which we obtain a double category of open factor graphs $\FG$, doubly opfibred over $\CCospan(\FinSet)$, by applying the double Grothendieck construction.
This gives $\FG$ the following structure:
\begin{enumerate}
\item Its objects (0-cells) are pairs $(A,\alpha)$ where $A$ is a finite set and $\alpha$ is an object of $F_0A$ --- \textit{i.e.}, an interface in $\Ca$ on $A$.
\item A vertical 1-cell $(A,\alpha)\to(A',\alpha')$ is a pair $(f,\varphi)$ of a map $f:A\to A'$ and a deterministic transformation of interfaces $\varphi:f_*\alpha\to\alpha'$ --- where $f_* = F_0(f)$ so that $\varphi$ is a morphism in $F_0A'$.
\item A horizontal 1-cell $(A,\alpha)\nrightarrow(B,\beta)$ is a quintuple $(X,a,b,\chi,p)$ where $\bigl(A\xrightarrow{a}X\xleftarrow{b}B\bigr)$ is a cospan of finite sets, $\chi$ is an interface on $X$, and $p$ is a factor on $\chi$ --- \textit{i.e.}, $(\chi,p)$ is an object of $\int\underline{F_1}(X,a,b)$ --- such that $\pi_a(\chi,p) = \alpha$ and $\pi_b(\chi,p) = \beta$.
\item A 2-cell as in the diagram
  \begin{tikzcd}[cramped,column sep=small,row sep=tiny]
	{(A,\alpha)} && {(B,\beta)} \\
	& {(f,\varphi)} \\
	{(A',\alpha')} && {(B',\beta')}
	\arrow["{(X,a,b,\chi,p)}", "\shortmid"{marking}, from=1-1, to=1-3]
	\arrow["{(X',a',b',\chi',p')}"', "\shortmid"{marking}, from=3-1, to=3-3]
	\arrow["{(f^l,\varphi^l)}"', from=1-1, to=3-1]
	\arrow["{(f^r,\varphi^r)}", from=1-3, to=3-3]
  \end{tikzcd}
  is a pair $(f,\varphi)$ where $f:X\to X'$ is a map of finite sets and $\varphi$ is a transformation of interfaces $f_*\chi\to\chi'$, such that $f_* p = \varphi^* p'$, $\pi_{a'}(\varphi) = \varphi^l$, $\pi_{b'}(\varphi) = \varphi^r$ (with $\varphi^l$ and $\varphi^r$ deterministic), and such that $(f,f^l,f^r)$ constitutes a morphism of cospans, meaning that $a'\circ f^l = f\circ a$ and $b'\circ f^r = f\circ b$.
\item Vertical 1-cells $(f,\varphi):(A,\alpha)\to(A',\alpha')$ and $(f',\varphi'):(A',\varphi')\to(A'',\varphi'')$ compose to yield $(f'\circ f, \varphi'\circ f'_*\varphi):(A,\alpha)\to(A'',\alpha'')$; the vertical composition of 2-cells is similar.
  The vertical identity 1-cell on $(A,\alpha)$ is the pair $(\id_A,\id_\alpha)$ of identity morphisms $\id_A:A\to A$ and $\id_\alpha:\alpha\to\alpha$.
  We denote vertical composition by $\circ$.
\item Finally, horizontal 1-cells (and 2-cells, horizontally) compose by cospan- and copy-composition, as described in \secref{sec:copy-comp}.
  Horizontal identities are given by the external unit $\eta$.
  We denote horizontal composition by $\diamond$. 
\end{enumerate}

All told, this (plus the details in Appendix~\ref{sec:F}) yields the following theorem.

\begin{thm}[Open factor graphs]
  Given a symmetric monoidal copy discard category $(\Ca,I,\otimes)$, the foregoing data define a pseudo double category $\FG$ doubly opfibred over $\CCospan(\FinSet)$.
\end{thm}

\begin{rmk}
  Since each of the elements of the construction of $\FG$ are functorial, we expect the construction itself to be functorial, so that a functor $\Ca\to\Ca'$ yields a double functor $\FG_{\Ca}\to\FG_{\Ca'}$.
  But we have not established this, and leave it to future work.
\end{rmk}

In other future work, we hope to make use of $\FG$ to study the compositionality of inference on factor graphs, and to study the formal relationships between directed and undirected models.

\subsection{On a graphical calculus}

We motivated the construction of $\FG$ with the graphic Figure~\ref{fig:fg-1}.
Decorated cospan categories as originally formulated inherit a graphical calculus from their status as hypergraph categories or undirected wiring diagram algebras.
Lacking a formal theory of double operads and their graphical calculi, we do not have a similarly formal graphical calculus for $\FG$; still, we expect that its horizontal bicategory (appropriately quotiented) would fit the undirected wiring diagram pattern, at the cost of the transformations of interfaces.
Inspired by this, and in an attempt not to pay that cost, in Appendix~\ref{sec:FG-graph} we (informally) propose a graphical calculus that matches and extends the non-compositional depictions in the statistics literature.

\section{References}
\printbibliography[heading=none]

\appendix

\section{Directed models}

\subsection{Kernels bifibration} \label{sec:krn-proofs}

\begin{proof}[Proof of Proposition \ref{prp:fib-res-comp}]
  By randomization \parencite[Lemma 2.22]{Kallenberg2002Foundations}, there is a measurable function $\rho:\Ibb\to\Ibb\times\Ibb$ such that $\rho_*\unif = \unif\otimes\unif$.
  Suppose $h,k$ have the types $(E,p)\xkto{h}(F,q)\xkto{k}(G,r)$, and suppose $\ovln{h}:E\to M_F$ and $\ovln{k}:F\to M_G$ represent them as indexed random variables, such that $\ovln{h}(x)_*\unif = h(x)$ and $\ovln{k}(y)_*\unif = k(y)$.
  The composite kernel $k\circ h$ may then be represented by the indexed random variable $kh:E\to M_G$ defined by $x\mapsto r\mapsto k\bigl(h(x)(\rho(r)_1)\bigr)(\rho(r)_2)$ where $(-)_1$ and $(-)_2$ represent the first and second projections out of the product; that is, we have $(k\circ h)(x) = kh(x)_*\unif$.

  From Proposition \ref{prop:fib-krn}, we know that if $x\in p[b]$, then $h(x)(s) \in q[b]$ for all $s\in\Ibb$.
  Similarly, if $y\in q[b]$, then $k(y)(s) \in r[b]$ for all $s\in\Ibb$.
  So, setting $y = h(x)(\rho(r)_1)$, we must have $k\bigl(h(x)(\rho(r)_1)\bigr)(s) \in r[b]$ for all $s\in\Ibb$, and in particular for $s = \rho(r)_2$.
  Thus $kh(x)(s) \in r[b]$, for all $s\in\Ibb$, as required.
\end{proof}

\begin{proof}[Proof of Proposition \ref{prp:K-sig-del-adj}]
  We need to exhibit an isomorphism
\[\begin{tikzcd}
	{\mathcal{K}_B(p,\Delta_fq)} && {\mathcal{K}_C(\Sigma_fp,q)}
	\arrow["{(-)^\flat}"{description}, curve={height=-12pt}, from=1-1, to=1-3]
	\arrow["{(-)^\sharp}"{description}, curve={height=-12pt}, from=1-3, to=1-1]
\end{tikzcd}\]
  natural in $p$ and $q$.

  Suppose then that we have $\beta:p\kto\Delta_fq$ in $\Ka_B$ and $\gamma:\Sigma_fp\kto q$ in $\Ka_C$ as in the following diagrams:
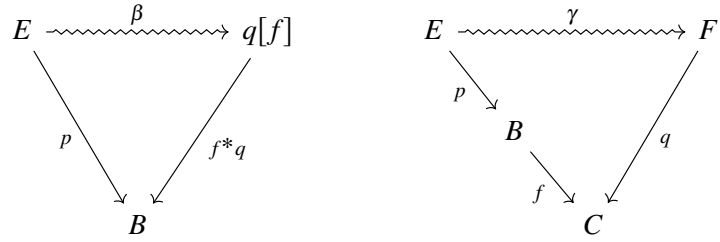
\[\begin{tikzcd}[column sep=small]
	E &&&& {q[f]} &&& E &&&& F \\
	&&&&&&&& B \\
	&& B &&&&&&& C
	\arrow["p"', from=1-1, to=3-3]
	\arrow["\beta", squiggly, from=1-1, to=1-5]
	\arrow["{f^*q}", from=1-5, to=3-3]
	\arrow["p"', from=1-8, to=2-9]
	\arrow["f"', from=2-9, to=3-10]
	\arrow["q", from=1-12, to=3-10]
	\arrow["\gamma", squiggly, from=1-8, to=1-12]
\end{tikzcd}\]
  Let $\ovln{\beta}:E\to\Ibb\to q[f]$ represent $\beta$ as an indexed random variable, and let $\pi_F$ be the projection $q[f]\to F$.
  Then we may define $\beta^\flat$ as represented by the map $\ovln{\beta^\flat}:E\to\Ibb\to F$ defined by $\ovln{\beta^\flat}(x)(r) = \pi_F(\beta(x)(r))$.

  Dually, define $\gamma^\sharp:p\kto\Delta_fq$ by the representative $\ovln{\gamma^\sharp}:E\to\Ibb\to q[f]:x\mapsto r\mapsto u_\gamma(x,r)$ where $u_\gamma$ is obtained from the universal property of the pullback in the following diagram (as the dashed arrow)
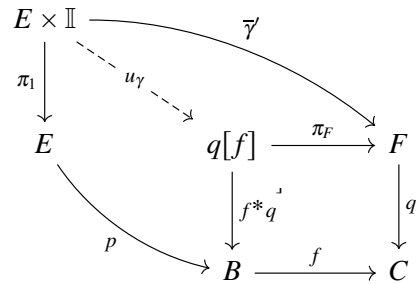
\[\begin{tikzcd}[sep=scriptsize]
	{E\times\Ibb} \\
	\\
	E && {q[f]} && F \\
	\\
	&& B && C
	\arrow["{\pi_1}"', from=1-1, to=3-1]
	\arrow["{u_\gamma}"{description}, dashed, from=1-1, to=3-3]
	\arrow["{\ovln{\gamma}'}", curve={height=-18pt}, from=1-1, to=3-5]
	\arrow["p"', curve={height=12pt}, from=3-1, to=5-3]
	\arrow["{\pi_F}", from=3-3, to=3-5]
	\arrow["{f^*q}", from=3-3, to=5-3]
	\arrow["\lrcorner"{anchor=center, pos=0.125}, draw=none, from=3-3, to=5-5]
	\arrow["q", from=3-5, to=5-5]
	\arrow["f", from=5-3, to=5-5]
\end{tikzcd}\]
  where $\pi_1$ is the projection out of the product and $\ovln{\gamma}':E\times\Ibb\to F$ is the uncurrying of the representative $\ovln{\gamma}:E\to\Ibb\to F$ of $\gamma$.

  From these definitions, it is clear that $\beta^\flat$ and $\gamma^\sharp$ are well-defined morphisms in $\Ka$, in the sense that the relevant triangles commute, so it remains to check the natural isomorphism.
  That the definitions yield an isomorphism of hom sets is similarly immediate.
  (On the one hand, let $\gamma = \beta^\flat$ in the diagram for $u_\gamma$ above, and observe that $\gamma^\sharp$ must equal $\beta$.
  On the other, let $\beta = \gamma^\sharp$ and observe that by construction $\beta^\flat$ must equal $\gamma$.)

  It therefore only remains to demonstrate naturality.
  Since $(-)^\flat$ and $(-)^\sharp$ are mutually inverse, we need only demonstrate naturality for one; we choose $(-)^\sharp$.
  This means showing that the following square commutes for all kernels $\varphi,\psi$ of appropriate type:
\[\begin{tikzcd}
	{\mathcal{K}_C(\Sigma_fp,q)} && {\mathcal{K}_B(p,\Delta_fq)} \\
	\\
	{\mathcal{K}_C(\Sigma_fo,r)} && {\mathcal{K}_B(o,\Delta_fr)}
	\arrow["{(-)^\sharp_{p,q}}", from=1-1, to=1-3]
	\arrow["{(-)^\sharp_{o,r}}", from=3-1, to=3-3]
	\arrow["{\mathcal{K}_C(\Sigma_f\varphi,\psi)}"', from=1-1, to=3-1]
	\arrow["{\mathcal{K}_B(\varphi,\Delta_f\psi)}", from=1-3, to=3-3]
\end{tikzcd}\]
  Since the hom functors $\Ka_C$ and $\Ka_B$ act by pre- and post-composition, this requires establishing that
  $$ o\xkto{\varphi}p\xkto{\gamma^\sharp}\Delta_f\,q\xkto{\Delta_f\,\psi}\Delta_f\,r
  = \left( \Sigma_f\,o\xkto{\Sigma_f\,\varphi}\Sigma_f\,p\xkto{\gamma}q\xkto{\psi}r \right)^{\sharp} \; . $$

  Let the composite on the right-hand side be denoted by $g$.
  Randomization gives us a function $\varrho:\Ibb\to\Ibb\times\Ibb\times\Ibb$ such that $\varrho_*\unif = \unif\otimes\unif\otimes\unif$.
  Let $\varrho(s) = (s_1,s_2,s_3)$.
  A representative $\ovln{g}$ of $g$ is given by $\ovln{g}(x)(s) = \ovln\psi\bigl(\ovln\gamma(\ovln\varphi(x)(s_1))(s_2)\bigr)(s_3)$ where $\ovln\psi$, $\ovln\gamma$, and $\ovln\varphi$ are representatives of $\psi$, $\gamma$, and $\varphi$ respectively.
  $g^\sharp$ is then represented by the currying of the dashed map $u_g$ in the following diagram:
\[\begin{tikzcd}[cramped,sep=scriptsize]
	{E\times\Ibb} \\
	\\
	E && {r[f]} && G \\
	\\
	&& B && C
	\arrow["{\pi_1}"', from=1-1, to=3-1]
	\arrow["{u_g}"{description}, dashed, from=1-1, to=3-3]
	\arrow["{\ovln{g}'}", curve={height=-18pt}, from=1-1, to=3-5]
	\arrow["p"', curve={height=12pt}, from=3-1, to=5-3]
	\arrow["{\pi_G}", from=3-3, to=3-5]
	\arrow["{f^*r}", from=3-3, to=5-3]
	\arrow["\lrcorner"{anchor=center, pos=0.125}, draw=none, from=3-3, to=5-5]
	\arrow["r", from=3-5, to=5-5]
	\arrow["f", from=5-3, to=5-5]
\end{tikzcd}\]

  Let the composite on the left-hand side (above) be denoted by $h$.
  A representative $\ovln{h}$ of $h$ is given by
  \begin{align*}
    \ovln{h}(x)(s)
    &= \ovln{\psi[f]}\bigl(\ovln{\gamma^\sharp}(\ovln\phi(x)(s_1))(s_2)\bigr)(s_3) \\
    &= \ovln{\psi[f]}\bigl(u_\gamma(\ovln\phi(x)(s_1), s_2)\bigr)(s_3) \\
    &= \ovln\psi\bigl(\pi_F(u_\gamma(\ovln\phi(x)(s_1), s_2))\bigr)(s_3)
  \end{align*}
  so we need to verify that $\pi_F(u_\gamma(\ovln\phi(x)(s_1), s_2)) = \ovln\gamma(\ovln\varphi(x)(s_1))(s_2)$.
  We have $\pi_F(u_\gamma(y,t)) = \ovln\gamma(y)(t)$ by construction, so $\ovln{h}(x)(s) = \ovln{g^\sharp}(x)(s)$.
  This establishes that $\Sigma_f \dashv \Delta_f$.
\end{proof}

\begin{proof}[Proof of Proposition \ref{prp:K-b-c}]
  The functor $\Sigma_\rho\,\Delta_\pi : \Ka_E\to\Ka_F$ sends $X\xto{\alpha}E$ to $\pi^*X\xto{\Delta_\pi\,\alpha}P\xto{\rho}F$; fibrewise, by adjointness, we have $\Sigma_\rho\,\Delta_\pi(\alpha)[i] \cong \alpha[\pi\circ\rho^*i]$ and, on morphisms $f$, we have $\Sigma_\rho\,\Delta_\pi(f)[i] = f[\pi\circ\rho^*i]$.
  On the other side, the functor $\Delta_q\,\Sigma_p$ sends $\alpha$ to $q^*X\xto{\Delta_q\,\alpha}q^*E\xto{\Delta_q\,p}F$; fibrewise, we have $\Delta_q\,\Sigma_p(\alpha)[i] \cong \alpha[p^*(q\circ i)]$ on objects and $\Delta_q\,\Sigma_p(f)[i] = f[p^*(q\circ i)]$ on morphisms.
  It therefore suffices to show that $p^*(q\circ i) = \pi\circ\rho^*i$ in $\QBS$.
  This follows from the universal property of the pullback, which implies that the pasting of two pullback squares is again a pullback square:
\[\begin{tikzcd}
	\bullet & P & E \\
	I & F & B
	\arrow["i"', from=2-1, to=2-2]
	\arrow["q"', from=2-2, to=2-3]
	\arrow["p", from=1-3, to=2-3]
	\arrow["\pi", from=1-2, to=1-3]
	\arrow["\rho"', from=1-2, to=2-2]
	\arrow[from=1-1, to=2-1]
	\arrow["{\rho^*i}", from=1-1, to=1-2]
	\arrow["{p^*(q\circ i)}", curve={height=-18pt}, from=1-1, to=1-3]
	\arrow["\lrcorner"{anchor=center, pos=0.125}, draw=none, from=1-1, to=2-2]
	\arrow["\lrcorner"{anchor=center, pos=0.125}, draw=none, from=1-2, to=2-3]
\end{tikzcd}\]
  The outer and inner squares are thus both pullbacks of $p$ along $q\circ i$, and hence we must have $p^*(q\circ i) = \pi\circ\rho^*i$.
  This yields an isomorphism $\Sigma_\rho\,\Delta_\pi(\alpha) \cong \Delta_q\,\Sigma_p(\alpha)$ rather than an equality because universal constructions are only unique up to natural isomorphism; and thus the naturality of this isomorphism follows similarly from universality.
\end{proof}

\subsection{Pull-push sections} \label{apdx:pull-push}

\begin{proof}[Proof of Theorem \ref{thm:S}]
  Much of the structure is simplified, owing to the trivial vertical decorations and the decoration of the apices by sets.
  Thus, for example, each component $\mu_{m,n}$ or $\eta_E$ of the pseudo natural transformations $\mu$ or $\eta$ is only a function (discrete functor).
  In the main text, the functoriality of $S_0$ and $S_1$ is justified.
  So here we will only justify the (pseudo) naturality of $\mu$ and $\eta$ and their (pseudo) associativity and unitality.

  Suppose then that $f:m\to m'$ and $g:n\to n'$ are horizontally composable morphisms of spans (\textit{i.e.}, such that $g_l=f_r$) from $m=\bigl(A\xleftarrow{a}E\xto{b}B\bigr)$ and $n=\bigl(B\xleftarrow{\bar{b}}F\xto{c}C\bigr)$ to $m'=\bigl(A'\xleftarrow{a'}E'\xto{b'}B'\bigr)$ and $n'=\bigl(B'\xleftarrow{\bar{b}'}F\xto{c'}C'\bigr)$.
  To validate the pseudo naturality of $\mu$ is to confirm the commutativity of the square
  \[\begin{tikzcd}[cramped,sep=small]
	{S_1(m')\times S_1(n')} && {S_1(n'\diamond m')} \\
	\\
	{S_1(m)\times S_1(n)} && {S_1(n\diamond m)}
	\arrow["{\mu_{m',n'}}", from=1-1, to=1-3]
	\arrow["{S_1(f)\times S_1(g)}"', from=1-1, to=3-1]
	\arrow["{S_1(g\diamond f)}", from=1-3, to=3-3]
	\arrow["{\mu_{m,n}}"', from=3-1, to=3-3]
  \end{tikzcd} \; .\]
  up to natural isomorphism.
  The top-right composite evaluates to
  \begin{align*}
    S_1(g\diamond f)\bigl(\Sigma_{a'}\,\Delta_{b'}(\tau')\circ\sigma'\bigr)
    &= \Delta_{f_l}\bigl(\Sigma_{a'}\,\Delta_{b'}(\tau')\circ\sigma'\bigr) & \text{(by definition)} \\
    &= \Delta_{f_l}\,\Sigma_{a'}\,\Delta_{b'}(\tau') \circ \Delta_{f_l}(\sigma') & \text{(by functoriality)}
  \end{align*}
  while the left-bottom composite evaluates to
  \begin{align*}
    \Sigma_a\,\Delta_b\bigl(S_1(g)(\tau')\bigr)\circ S_1(f)(\sigma') &= \Sigma_a\,\Delta_b\,\Delta_{g_l}(\tau')\circ \Delta_{f_l}(\sigma') & \text{(by definition}) \\
    &= \Sigma_a\,\Delta_f\,\Delta_{b'}(\tau')\circ \Delta_{f_l}(\sigma') & \text{(since $g_l\circ b = b'\circ f$)} \\
    &\cong \Delta_{f_l}\,\Sigma_{a'}\,\Delta_{b'}(\tau')\circ \Delta_{f_l}(\sigma') & \text{(by Beck-Chevalley)}
  \end{align*}
  and so $S_1(g\diamond f)\circ\mu_{m',n'} \cong \mu_{m,n}\circ\bigl(S_1(f)\times S_1(g)\bigr)$ as required.

  Validating the pseudo naturality of $\eta$ means, for any $f:E\to F$ in $\Ba$, validating that $\eta_F \cong S_1(f)\circ \eta_E$, \textit{i.e.}, $\eta_F \cong \Delta_f\,\eta_E$.
  Since $\eta_E$ is the identity on $\iota E$ and $\Delta_f$ is a functor, it must preserve identities.
  So in fact $\eta$ is strictly natural (just as pull-push copy-composition is strictly unital, as we confirm below).
  
  To validate pseudo associativity, suppose $m$ and $n$ are spans as before and $o=\bigl(C\xleftarrow{\bar{c}}G\xto{d}D\bigr)$.
  We need to show that $\mu_{n\diamond m,o}\bigl(\mu_{m,n}(\sigma,\tau),\rho\bigr) \cong \mu_{m,o\diamond n}\bigl(\sigma,\mu_{n,o}(\tau,\rho)\bigr)$ for all appropriately typed sections $\sigma,\tau,\rho$.
  This obtains as follows:
  \begin{align*}
    \mu_{n\diamond m,o}\bigl(\mu_{m,n}(\sigma,\tau),\rho\bigr)
    &= \mu_{n\diamond m,o}\bigl(\Sigma_a\,\Delta_b(\tau)\circ\sigma,\rho\bigr) & \text{(by definition)} \\
    &= \Sigma_a\,\Sigma_{\pi_E}\,\Delta_{\pi_F}\,\Delta_c(\rho) \circ \Sigma_a\,\Delta_b(\tau)\circ\sigma & \text{(by definition)} \\
    &\cong \Sigma_a\,\Delta_b\,\Sigma_{\bar{b}}\,\Delta_c(\rho) \circ \Sigma_a\,\Delta_b(\tau) \circ \sigma & \text{(by Beck-Chevalley)} \\
    &= \Sigma_a\,\Delta_b\bigl(\Sigma_{\bar{b}}\,\Delta_c(\rho) \circ \tau\bigr) \circ \sigma & \text{(by functoriality)} \\
    &= \mu_{m,o\diamond n}\bigl(\sigma, \Sigma_{\bar{b}}\,\Delta_c(\rho)\circ\tau\bigr) & \text{(by definition)} \\
    &= \mu_{m,o\diamond n}\bigl(\sigma,\mu_{n,o}(\tau,\rho)\bigr) & \text{(by definition)}
  \end{align*}
  where $\pi_E$ and $\pi_F$ are the projections out of the pullback of $\bar{b}$ along $b$.

  Finally, strict unitality demands that $\mu_{\id_B,n}(\eta_B,\tau) = \tau$ and $\mu_{m,\id_B}(\sigma,\eta_B) = \sigma$.
  We have
  $$ \mu_{\id_B,n}(\eta_B,\tau) = \Sigma_{\id_B}\,\Delta_{\id_B}(\tau)\circ\eta_B = \tau\circ\id_{\iota_B} = \tau $$
  and
  $$ \mu_{m,\id_B}(\sigma,\eta_B) = \Sigma_a\,\Delta_b(\eta_B)\circ\sigma = \Sigma_a\,\Delta_b(\id_{\iota_B})\circ\sigma = \Sigma_a(\id_{\iota_E})\circ\sigma = \id_{\Sigma_a\,\iota_E}\circ\sigma = \sigma $$
  as required.
\end{proof}

\section{Factor graphs} \label{sec:F}

\subsection{Constructing the horizontal decorations} \label{sec:F1}

The apex decoration has a peculiar form, with one category of decorating data (the interfaces) further decorated by a family of others (the factors).
This is because the apex decorations are themselves obtained by a Grothendieck construction, defined functorially for each finite set.
The first step in this construction is to define the interface categories.
Since we have the cospans determining which parts of the interfaces are exposed for factor composition, there may now be some parts which remain `latent', and transformations of these parts are not constrained to be comonoid homomorphisms.
That is, we can allow the transformations of the latent parts of the interfaces to be arbitrary morphisms in $\Ca$, as long as the exposed parts remain constrained as before\footnote{
It is this part of the construction that means we cannot use Fong's original decorated cospans: here, the decoration on the apex necessarily depends on the legs, in order that we have comonoid homomorphisms in the right places.}.

This line of thought formalizes as follows.
Given a cospan $m:\bigl(A\xrightarrow{a} X\xleftarrow{b} B\bigr)$, we first define an enlarged version of $F_0X$, not totally restricted to comonoid homomorphisms. 
We denote this enlargement by $\tilde{F_0}m$ and construct it as the following limit of categories, so that we do still have comonoid homomorphisms over the legs:
\[\begin{tikzcd}
	{\mathbf{Cat}\bigl(\mathsf{disc}\,A,\mathbf{Comon}(\mathcal{C})\bigr)} && {\tilde{F_0}m} && {\mathbf{Cat}\bigl(\mathsf{disc}\,B,\mathbf{Comon}(\mathcal{C})\bigr)} \\
	\\
	{\mathbf{Cat}(\mathsf{disc}\,A,\mathcal{C})} && {\mathbf{Cat}(\mathsf{disc}\,X,\mathcal{C})} && {\mathbf{Cat}(\mathsf{disc}\,B,\mathcal{C})}
	\arrow["{\mathbf{Cat}(\mathsf{disc}\,a,\mathcal{C})}", from=3-3, to=3-1]
	\arrow["{\mathbf{Cat}(\mathsf{disc}\,b,\mathcal{C})}"', from=3-3, to=3-5]
	\arrow[hook, from=1-1, to=3-1]
	\arrow[hook, from=1-5, to=3-5]
	\arrow[dashed, from=1-3, to=1-1]
	\arrow[dashed, from=1-3, to=3-3]
	\arrow[dashed, from=1-3, to=1-5]
\end{tikzcd}\]
Note that, because the left and right vertical functors in this diagram are embeddings, so is the induced functor $\tilde{F_0}m\to\Cat(\disc X,\Ca)$.
We can therefore understand $\tilde{F_0}m$ as the wide subcategory of $\Cat(\disc X,\Ca)$ whose morphisms are those $X$-indexed families of morphisms of $\Ca$ whose $A$- and $B$-components are comonoid homomorphisms.

The next step is to decorate $\tilde{F_0}m$ with factors.
Since $\Ca$ is a category, factors on an interface $\chi:\disc X\to\Ca$ form a set $\Ca(\chi^\otimes,I)$.
We thus define a presheaf $\underline{F_1}m:\tilde{F_0}m\op\to\Set$ as follows.
On objects (interfaces) $\chi$, we define $\underline{F_1}m(\chi)$ to be $\Ca(\chi^\otimes,I)$.
On morphisms of interfaces $\varphi:\chi\to\chi'$, we set $\underline{F_1}m(\varphi):\underline{F_1}m(\chi')\to\underline{F_1}m(\chi)$ to be the precomposition functor $\Ca(\varphi^\otimes,I):\Ca({\chi'}^\otimes,I)\to\Ca(\chi^\otimes,I)$ and denote it by $\varphi^*$.
This yields a well-defined presheaf by the functoriality of $\otimes$ and of precomposition.
The apex decoration returned by $F_1$ on $m$ is then given by the 1-categorical Grothendieck construction of $\underline{F_1}m$ (its category of elements), denoted $\int\underline{F_1}m$, whose objects and morphisms are as described in the main text.

This tells us how to decorate the apices of cospans --- but, given the cospan $m$, $F_1$ needs to return a whole span of categories $F_0A \xleftarrow{\pi_a} \int\underline{F_1}m \xrightarrow{\pi_b} F_0B$.
We obtain the legs of this span through the universal properties of the construction of $\int\underline{F_1}m$.
First, because $\int\underline{F_1}m$ is obtained by a Grothendieck construction, it is equipped with a projection functor $\pi_m : \int\underline{F_1}m \to \tilde{F_0}m$ which acts by `forgetting' the factors; \textit{i.e.} $\pi_m(\chi,f) = \chi$.
Next, the construction of $\tilde{F_0}m$ as a limit means that it is equipped with canonical projection functors to $F_0A$ and $F_0B$:
\[\begin{tikzcd}
	{F_0A=\mathbf{Cat}\bigl(\mathsf{disc}\,A,\mathbf{Comon}(\mathcal{C})\bigr)} & {\tilde{F_0}m} & {\mathbf{Cat}\bigl(\mathsf{disc}\,B,\mathbf{Comon}(\mathcal{C})\bigr)=F_0B}
	\arrow[dashed, from=1-2, to=1-3]
	\arrow[dashed, from=1-2, to=1-1]
\end{tikzcd}\]
Therefore, we obtain the requisite span of categories
$F_1m := \bigl(F_0A \xleftarrow{\pi_a} \int\underline{F_1}m \xrightarrow{\pi_b} F_0B\bigr)$
as the composition of this span of limit projections with $\pi_m$:
\[\begin{tikzcd}
	&& {\displaystyle\int\underline{F_1}m} \\
	\\
	{F_0A} && {\tilde{F_0}m} && {F_0B}
	\arrow[dashed, from=3-3, to=3-5]
	\arrow[dashed, from=3-3, to=3-1]
	\arrow["{\pi_m}", from=1-3, to=3-3]
	\arrow["{\pi_a}"', curve={height=18pt}, from=1-3, to=3-1]
	\arrow["{\pi_b}", curve={height=-18pt}, from=1-3, to=3-5]
\end{tikzcd}\]

We now need to define the action of $F_1$ on morphisms of cospans.
Suppose that $m' := \bigl(A'\xleftarrow{a'}X'\xrightarrow{B'}\bigr)$ is another cospan.
A morphism $\phi:m\to m'$ is a triple $(\phi_l,\phi,\phi_r)$ of maps of finite sets making the following diagram commute:
\[\begin{tikzcd}
	A && X && B \\
	\\
	{A'} && {X'} && {B'}
	\arrow["a", from=1-1, to=1-3]
	\arrow["{a'}"', from=3-1, to=3-3]
	\arrow["{\phi_l}"', from=1-1, to=3-1]
	\arrow["\phi", from=1-3, to=3-3]
	\arrow["b"', from=1-5, to=1-3]
	\arrow["{\phi_r}", from=1-5, to=3-5]
	\arrow["{b'}", from=3-5, to=3-3]
\end{tikzcd}\]
We know from the treatment above that $F_0$ is functorial, and this yields the action of $F_1$ on the left and right components $\phi_l$ and $\phi_r$.
Moreover, the construction of $\int\underline{F_1}m$ is functorial: first, the construction of $\tilde{F_0}m$ is functorial in $m$, by the functoriality of taking limits.
Likewise, the Grothendieck construction is functorial, yielding a functor $\int\underline{F_1} : \FinSet^{\{\bullet\to\bullet\leftarrow\bullet\}} \to \Cat$.
The following diagram then commutes by construction:
\[\begin{tikzcd}
	{F_0A} && {\displaystyle\int\underline{F_1}m} && {F_0B} \\
	\\
	{F_0A'} && {\displaystyle\int\underline{F_1}m'} && {F_0B'}
	\arrow["{\pi_a}"', from=1-3, to=1-1]
	\arrow["{\pi_{a'}}", from=3-3, to=3-1]
	\arrow["{F_0\phi_l}"', from=1-1, to=3-1]
	\arrow["{\int\underline{F_1}\phi}", from=1-3, to=3-3]
	\arrow["{\pi_b}", from=1-3, to=1-5]
	\arrow["{F_0\phi_r}", from=1-5, to=3-5]
	\arrow["{\pi_{b'}}"', from=3-3, to=3-5]
\end{tikzcd}\]
Consequently, we define the action of $F_1$ on the morphism of cospans $\phi$ to return the morphism of spans $F_1\phi$ whose components are $(F_0\phi_l,\int\underline{F_1}\phi,F_0\phi_r)$, and we see that this action is functorial by construction.

\subsection{The horizontal composition} \label{sec:F-mu}

$\mu$ must be a pseudo natural transformation filling the following diagram:
\[\begin{tikzcd}
	{\mathbf{FinSet}^{\{\bullet\to\bullet\leftarrow\bullet\}}\times_{\mathbf{FinSet}}\mathbf{FinSet}^{\{\bullet\to\bullet\leftarrow\bullet\}}} && {\mathbf{FinSet}^{\{\bullet\to\bullet\leftarrow\bullet\}}} \\
	\\
	{\mathbf{Cat}^{\{\bullet\leftarrow\bullet\to\bullet\}}\times_{\mathbf{Cat}}\mathbf{Cat}^{\{\bullet\leftarrow\bullet\to\bullet\}}} && {\mathbf{Cat}^{\{\bullet\leftarrow\bullet\to\bullet\}}}
	\arrow["{\diamond_{\Bbb{C}\mathbf{ospan}(\mathbf{FinSet})}}", from=1-1, to=1-3]
	\arrow["{F_1\times F_1}"', from=1-1, to=3-1]
	\arrow["{\diamond_{\Bbb{S}\mathbf{pan}(\mathbf{Cat})}}"', from=3-1, to=3-3]
	\arrow["{F_1}", from=1-3, to=3-3]
	\arrow["\mu", shorten <=35pt, shorten >=35pt, Rightarrow, from=3-1, to=1-3]
\end{tikzcd}\]
The external composition $\diamond_{\CCospan(\FinSet)}$ acts on composable cospans by pushout, taking $m:\bigl(A\xrightarrow{a}X\xleftarrow{b}B\bigr)$ and $n:\bigl(B\xrightarrow{b'}Y\xleftarrow{c}C\bigr)$ to the outer cospan in the pushout diagram
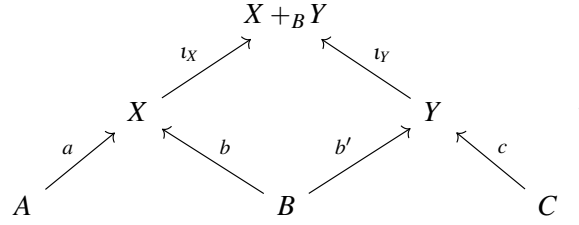
\[\begin{tikzcd}
	&& {X+_BY} \\
	& X && Y \\
	A && B && C
	\arrow["a", from=3-1, to=2-2]
	\arrow["b"', from=3-3, to=2-2]
	\arrow["{b'}", from=3-3, to=2-4]
	\arrow["c"', from=3-5, to=2-4]
	\arrow["{\iota_X}", from=2-2, to=1-3]
	\arrow["{\iota_Y}"', from=2-4, to=1-3]
\end{tikzcd} \; .\]
(Psuedo functoriality of $\diamond_{\CCospan(\FinSet)}$ follows from the universal properties of colimits.)
The composite pseudo functor along the top and right of the above diagram therefore acts on $m$ and $n$ to yield the span of categories
$$F_0A \xleftarrow{\pi_{\iota_X\circ a}} \int\underline{F_1}(n\diamond m) \xrightarrow{\pi_{\iota_Y\circ c}} F_0C \; .$$
Dually, $\diamond_{\SSpan(\Cat)}$ acts on composable spans by pullback, so that the composite pseudo functor along the left and bottom of the diagram acts to yield the outer span in the following diagram:
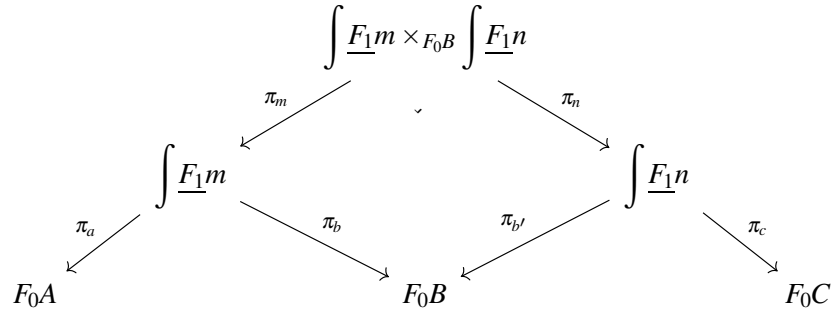
\[\begin{tikzcd}
	&& {\displaystyle\int\underline{F_1}m\times_{F_0B}\int\underline{F_1}n} \\
	& {\displaystyle\int\underline{F_1}m} && {\displaystyle\int\underline{F_1}n} \\
	{F_0A} && {F_0B} && {F_0C}
	\arrow["{\pi_m}"', from=1-3, to=2-2]
	\arrow["{\pi_n}", from=1-3, to=2-4]
	\arrow["{\pi_b}", from=2-2, to=3-3]
	\arrow["{\pi_{b'}}"', from=2-4, to=3-3]
	\arrow["\lrcorner"{anchor=center, pos=0.125, rotate=-45}, draw=none, from=1-3, to=3-3]
	\arrow["{\pi_a}"', from=2-2, to=3-1]
	\arrow["{\pi_c}", from=2-4, to=3-5]
\end{tikzcd}\]
Note that, here and throughout this section, $\pi_m$ and $\pi_n$ are the projections out of the pullback, rather than the Grothendieck fibrations of the previous section.

In both cases, the feet of the span are $F_0A$ and $F_0C$ on the left and right respectively; this must be the case for $F$ to be a well-defined double functor.
Consequently, the component of $\mu$ at $(m,n)$, denoted $\mu_{m,n}$, must be a functor $\int\underline{F_1}m\times_{F_0B}\int\underline{F_1}n\to\int\underline{F_1}(n\diamond m)$ making this diagram commute:
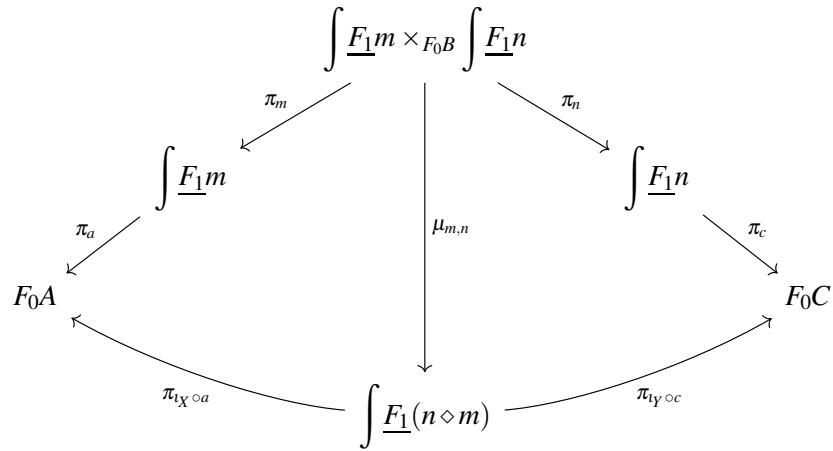
\[\begin{tikzcd}
	&& {\displaystyle\int\underline{F_1}m\times_{F_0B}\int\underline{F_1}n} \\
	& {\displaystyle\int\underline{F_1}m} && {\displaystyle\int\underline{F_1}n} \\
	{F_0A} &&&& {F_0C} \\
	&& {\displaystyle\int\underline{F_1}(n\diamond m)}
	\arrow["{\pi_n}", from=1-3, to=2-4]
	\arrow["{\pi_c}", from=2-4, to=3-5]
	\arrow["{\pi_m}"', from=1-3, to=2-2]
	\arrow["{\pi_a}"', from=2-2, to=3-1]
	\arrow["{\mu_{m,n}}", from=1-3, to=4-3]
	\arrow["{\pi_{\iota_X\circ a}}", curve={height=-12pt}, from=4-3, to=3-1]
	\arrow["{\pi_{\iota_Y\circ c}}"', curve={height=12pt}, from=4-3, to=3-5]
\end{tikzcd}\]
This means that $\mu_{m,n}$ takes a pair of factors over $X$ and $Y$, that are composable by agreeing over $B$, and returns a composite factor over $X+_BY$, all while preserving the exposed interfaces over $A$ and $C$.
We define $\mu_{m,n}$ to act by copy-composition, as follows.

First, notice that both $\int\underline{F_1}(n\diamond m)$ and $\int\underline{F_1}m\times_{F_0B}\int\underline{F_1}n$ are (discrete) fibrations; the former over $\tilde{F_0}(n\diamond m)$ by definition and the latter over $\tilde{F_0}m\times_{F_0B}\tilde{F_0}n$, by the universal property of the pullback:
\[\begin{tikzcd}
	& {\displaystyle\int\underline{F_1}m\times_{F_0B}\int\underline{F_1}n} \\
	{\displaystyle\int\underline{F_1}m} && {\displaystyle\int\underline{F_1}n} \\
	& {\tilde{F_0}m\times_{F_0B}\tilde{F_0}n} \\
	{\tilde{F_0}m} && {\tilde{F_0}n} \\
	& {F_0B}
	\arrow[from=3-2, to=4-1]
	\arrow[from=3-2, to=4-3]
	\arrow[from=4-1, to=5-2]
	\arrow[from=4-3, to=5-2]
	\arrow["{\pi_m}"', from=2-1, to=4-1]
	\arrow["{\pi_n}", from=2-3, to=4-3]
	\arrow[from=1-2, to=2-1]
	\arrow[from=1-2, to=2-3]
	\arrow[dashed, from=1-2, to=3-2]
	\arrow["\lrcorner"{anchor=center, pos=0.125, rotate=-45}, draw=none, from=3-2, to=5-2]
\end{tikzcd}\]
$\mu_{m,n}$ should therefore correspond to a morphism of fibrations.
By the Grothendieck construction \parencite[p.7]{Moeller2018Monoidal}, such a morphism is equivalent to a pair $(\mu^0_{m,n},\mu^1_{m,n})$ of a functor and a natural transformation as in the following diagram:
\[\begin{tikzcd}
	{\bigl(\tilde{F_0}m\times_{F_0B}\tilde{F_0}n\bigr)^{\,\mathrm{op}}} && {\tilde{F_0}(n\diamond m)^{\,\mathrm{op}}} \\
	\\
	& {\mathbf{Set}}
	\arrow["{\mu^{0\;\mathrm{op}}_{m,n}}", from=1-1, to=1-3]
	\arrow[""{name=0, anchor=center, inner sep=0}, "{\bigl(\underline{F_1}m, \underline{F_1}n\bigr)}"', from=1-1, to=3-2]
	\arrow[""{name=1, anchor=center, inner sep=0}, "{\underline{F_1}(n\diamond m)}", from=1-3, to=3-2]
	\arrow["{\mu^1_{m,n}}", shorten <=11pt, shorten >=11pt, Rightarrow, from=0, to=1]
\end{tikzcd}\]
To define the functor $\mu^0_{m,n}: \tilde{F_0}m \times_{F_0B} \tilde{F_0}n \to \tilde{F_0}(n\diamond m)$, note that the objects in its domain are pairs $(\chi,\gamma)$ of functors $\chi:\disc X\to\Ca$ and $\gamma:\disc Y\to\Ca$ that agree on $B$.
By the universal properties of the coproduct $X+Y$ and the pushout $X+_BY$, there must be copairings $[\chi,\gamma]:\disc(X+Y)\to\Ca$ and $[\chi,\gamma]_B:\disc(X+_BY)\to\Ca$ making the following diagram commute; in particular, $[\chi,\gamma]$ must factor through $[\chi,\gamma]_B$:
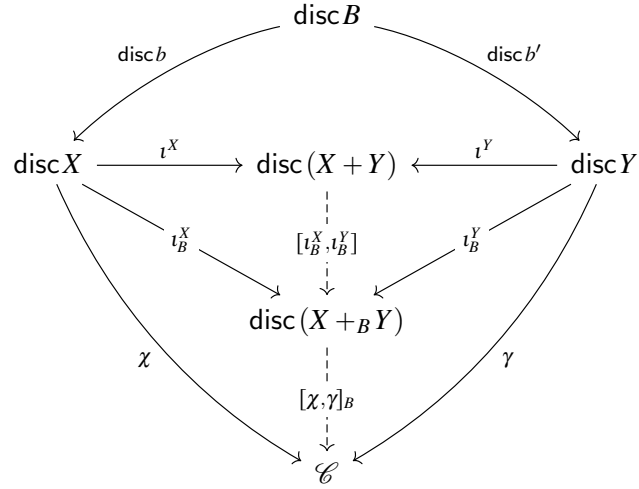
\[\begin{tikzcd}
	&& {\mathsf{disc}\,B} \\
	\\
	{\mathsf{disc}\,X} && {\mathsf{disc}\,(X+Y)} && {\mathsf{disc}\,Y} \\
	\\
	&& {\mathsf{disc}\,(X+_BY)} \\
	\\
	&& {\mathcal{C}}
	\arrow["\chi"', curve={height=18pt}, from=3-1, to=7-3]
	\arrow["\gamma", curve={height=-18pt}, from=3-5, to=7-3]
	\arrow["{\mathsf{disc}\,b}"', curve={height=12pt}, from=1-3, to=3-1]
	\arrow["{\mathsf{disc}\,b'}", curve={height=-12pt}, from=1-3, to=3-5]
	\arrow["{\iota^X_B}"{description}, from=3-1, to=5-3]
	\arrow["{\iota^Y_B}"{description}, from=3-5, to=5-3]
	\arrow["{[\chi,\gamma]_B}"{description}, dashed, from=5-3, to=7-3]
	\arrow["{\iota^X}", from=3-1, to=3-3]
	\arrow["{\iota^Y}"', from=3-5, to=3-3]
	\arrow["{[\iota^X_B,\iota^Y_B]}"{description}, dashed, from=3-3, to=5-3]
\end{tikzcd}\]
Since the image of $(\chi,\gamma)$ under $\mu^0_{m,n}$ must be a functor $\disc(X+_BY)\to\Ca$, we make the universal definition $\mu^0_{m,n}(\chi,\gamma) := [\chi,\gamma]_B$.
Functoriality follows from the functoriality of colimits.

The corresponding component of $\mu^1_{m,n}$ must then be a function
$$ \mu^1_{m,n}(\chi,\gamma) : \Ca(\chi^\otimes,I) \times \Ca(\gamma^\otimes,I) \to \Ca([\chi,\gamma]_B^\otimes,I) \; .$$
Note that $[\chi,\gamma]^\otimes \cong \chi^\otimes \otimes \gamma^\otimes$.
We can thus factor $\mu^1_{m,n}$ as
$$ \Ca(\chi^\otimes,I) \times \Ca(\gamma^\otimes,I) \xrightarrow{\otimes} \Ca(\chi^\otimes\otimes\gamma^\otimes,I) \xrightarrow{\sim} \Ca([\chi,\gamma]^\otimes,I) \to \Ca([\chi,\gamma]_B^\otimes,I)$$
so that, by Yoneda, we seek a morphism $\delta_{\chi,\gamma}:[\chi,\gamma]_B^\otimes \to [\chi,\gamma]^\otimes$ in $\Ca$.

It is here that copying finally enters.
For each $j:X+_BY$, copying induces a morphism
$$\delta_{\chi,\gamma}^j:[\chi,\gamma]_B(j) \to \bigotimes_{i:[\iota^X_B,\iota^Y_B]^{-1}(j)}\, [\chi,\gamma](i)$$
in $\Ca$ which is unique up to coassociativity.
Note that the codomain of this morphism is, by the pullback constraint, the monoidal product of $[\iota^X_B,\iota^Y_B]^{-1}(j)$-many copies of the same object $[\chi,\gamma]_B(j)$ so that this copying is well-defined; and when $j$ is not in the image of $B$, then this product is only unary, so that in this case, $\delta^j_{\chi,\gamma}$ is the corresonding identity morphism.
Then
$$\bigotimes_{j:X+_BY} \delta_{\chi,\gamma}^j : \bigotimes_{j:X+_BY} [\chi,\gamma]_B(j) \to \bigotimes_{j:X+_BY} \bigotimes_{i:[\iota^X_B,\iota^Y_B]^{-1}(j)} [\chi,\gamma](i)$$
has the requisite type, since $\otimes_{j:X+_BY}\, [\chi,\gamma]_B(j) = [\chi,\gamma]_B^\otimes$ and
$$\bigotimes_{j:X+_BY} \bigotimes_{i:[\iota^X_B,\iota^Y_B]^{-1}(j)} [\chi,\gamma](i) = \bigotimes_{i:X+Y} [\chi,\gamma](i) = [\chi,\gamma]^\otimes$$
both by definition.
Hence we define $\delta_{\chi,\gamma} := \otimes_{j:X+_BY}\;\delta_{\chi,\gamma}^j$.

We need to verify that $\delta_{\chi,\gamma}$ is natural in $\chi$ and $\gamma$.
Since identity morphisms are always natural, we only need to check the case that $[\iota^X_B,\iota^Y_B]^{-1}(j)$ has more than one element --- and this in turn only obtains for $j$ in the image of $B$.
By construction, morphisms of interfaces on those components are necessarily comonoid homomorphisms, and these are the morphisms for which $\delta^j_{\chi,\gamma}$ is natural.
Therefore, $\delta_{\chi,\gamma}$ is always natural in $\chi$ and $\gamma$, and since every factor of $\mu^1_{m,n}$ is accordingly natural, so is $\mu^1_{m,n}$ itself, as required.

This means that $(\mu^0_{m,n},\mu^1_{m,n})$ constitutes a well-defined morphism of indexed sets, thereby inducing a well-defined morphism $\mu_{m,n}$ of discrete fibrations.
Thus, given factors $(\chi,f)$ and $(\gamma,g)$ that agree over $B$, the external composition $\mu_{m,n}$ acts to return $\bigl([\chi,\gamma]_B,(f\otimes g)\circ\delta_{\chi,\gamma}\bigr)$.

We also need to define the unit $\eta$, a pseudo natural transformation as in the following diagram:
\[\begin{tikzcd}
	{\mathbf{FinSet}} && {\mathbf{FinSet}^{\{\bullet\rightarrow\bullet\leftarrow\bullet\}}} \\
	\\
	{\mathbf{Cat}} && {\mathbf{Cat}^{\{\bullet\leftarrow\bullet\rightarrow\bullet\}}}
	\arrow["{\mathsf{id}_{\mathbb{C}\mathbf{ospan}(\mathbf{FinSet})}}", from=1-1, to=1-3]
	\arrow["{\mathsf{id}_{\mathbb{S}\mathbf{pan}(\mathbf{Cat})}}", from=3-1, to=3-3]
	\arrow["{F_0}"', from=1-1, to=3-1]
	\arrow["{F_1}", from=1-3, to=3-3]
	\arrow["\eta", shorten <=15pt, shorten >=15pt, Rightarrow, from=3-1, to=1-3]
\end{tikzcd}\]
The component $\eta_X$ at each finite set $X$ is determined by a functor $\eta_X : F_0X \to \int\underline{F_1}(\id_X)$ as in the diagram
\[\begin{tikzcd}[cramped,sep=small]
	{F_0X} && {F_0X} && {F_0X} \\
	\\
	{F_0X} && {\displaystyle\int\underline{F_1}(\mathsf{id}_X)} && {F_0X}
	\arrow["{\pi_X}"', from=3-3, to=3-1]
	\arrow["{\pi_X}", from=3-3, to=3-5]
	\arrow[Rightarrow, no head, from=1-1, to=3-1]
	\arrow[Rightarrow, no head, from=1-5, to=3-5]
	\arrow[Rightarrow, no head, from=1-3, to=1-1]
	\arrow[Rightarrow, no head, from=1-3, to=1-5]
	\arrow["{\eta_X}"{description}, from=1-3, to=3-3]
\end{tikzcd}\]
where $\id_X$ denotes the identity cospan $X=X=X$ on $X$.
Note that because $\int\underline{F_1}(\mathsf{id}_X)$ is defined over the identity span, $\tilde{F_0}(\id_X) \cong F_0X$, and so $\pi_X$ simply forgets the factor decorations.
Moreover, this means that all components of morphisms of interfaces in $\tilde{F_0}(\id_X)$ are comonoid homomorphisms in $\Ca$.

We define $\eta_X$ using the counits of the comonoid structures on the interfaces on $X$.
That is to say, $\eta_X$ maps an interface $\chi:\disc X \to \Comon(\Ca)$ to the pair $(\chi, \epsilon_\chi)$, where here $\epsilon_\chi$ denotes the canonical discarding map $\chi^\otimes \to I$.
The functoriality of $\eta_X$ follows from the fact that all morphisms in $F_0X$ are comonoid homomorphisms, so if $\varphi$ is a morphism $\chi \to \chi'$ then $\epsilon_\chi = \varphi^*\epsilon_{\chi'}$.
Pseudo naturality of $\eta$ obtains likewise, and its unitality with respect to $\mu$ follows from the counitality of the comonoid counits.

At this final point, we should also validate the remaining laws that a lax double functor should satisfy: in particular, the associativity of $\mu$.
However, we shall at this point simply assert that associativity follows from the associativity of $\otimes$, the coassociativity of the relevant comonoid structures, and the universal properties of the (co)limits involved in the constructions, leaving a detailed examination of this claim to future exposition.

\subsection{Graphical calculus} \label{sec:FG-graph}

The relationship between decorated cospans and undirected wiring diagram algebras indicates that factor graphs should be understood as composing `operadically' (multicategorically): that is to say, by nesting; thus, within each factor of a factor graph may be hiding a whole other factor graph.
This formalizes the situation sometimes encountered in the literature --- \textit{e.g.} \textcite[Fig.20]{Koudahl2023Realising} --- in which a collection of factors is grouped into a composite factor, sometimes depicted by drawing an extra box around the collection.

To hew more closely to existing undirected wiring diagrams, as well as the standard string diagrams for monoidal categories, we will adopt a syntax in which factors are depicted with circular boxes, exposed variables by emanating wires, and transformation morphisms by rectangular boxes.
We will depict comonoid homomorphisms as `diamond' shaped boxes, both to indicate that they commute with horizontal composition $\diamond$ and also to suggest that they can pierce through the bubble enclosing a factor; this means we will attach wires to their corners.
Conversely, we will depict arbitrary morphisms with standard rectangular boxes, attaching wires to their sides.

\begin{figure}[H]
  \[\scalebox{0.8}{\begin{tikzpicture}
	\begin{pgfonlayer}{nodelayer}
		\node [style=dot-1cm] (0) at (0, 0) {$f$};
		\node [style=box-7mm] (1) at (0, 3) {$\psi$};
		\node [style=dia-7mm] (2) at (3, 0) {$\varphi$};
		\node [style=none] (3) at (-3, -3) {};
		\node [style=none] (4) at (0, 5.5) {};
		\node [style=none] (5) at (5.5, 0) {};
	\end{pgfonlayer}
	\begin{pgfonlayer}{edgelayer}
		\draw (0) to (1);
		\draw (0) to (2);
		\draw (0) to (3.center);
		\draw (1) to (4.center);
		\draw (2) to (5.center);
	\end{pgfonlayer}
\end{tikzpicture}}\]
  \caption{A composite factor $f'$.}
  \label{fig:fg-eg-1}
\end{figure}

Figure~\ref{fig:fg-eg-1} shows a simple example of a composite factor $f'$, constituted by another factor $f$ with three exposed variables, two of which are composed with transformations, one deterministic ($\varphi$) and one not ($\psi$).
Thinking statistically, we can interpret exposed variables (dangling wires) as representing \textit{observed random variables}.
We can draw a bubble around the whole of Figure~\ref{fig:fg-eg-1} to obtain a new factor with three unobserved variables; nonetheless these variables are observ\textit{able}, which we depict by terminating them with a black dot.
As a horizontal 1-cell in $\FG$, the resulting factor has the type $f_0:0\nrightarrow 0$, where $0$ denotes the empty set.
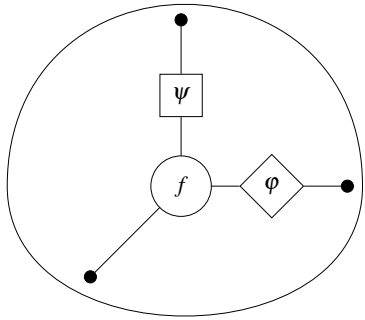
\[\scalebox{0.8}{\begin{tikzpicture}
	\begin{pgfonlayer}{nodelayer}
		\node [style=dot-1cm] (0) at (0, 0) {$f$};
		\node [style=box-7mm] (1) at (0, 3) {$\psi$};
		\node [style=dia-7mm] (2) at (3, 0) {$\varphi$};
		\node [style=bcopy] (3) at (-3, -3) {};
		\node [style=bcopy] (4) at (0, 5.5) {};
		\node [style=bcopy] (5) at (5.5, 0) {};
		\node [style=none] (6) at (-5.75, 0) {};
		\node [style=none] (7) at (6, 0) {};
	\end{pgfonlayer}
	\begin{pgfonlayer}{edgelayer}
		\draw (0) to (1);
		\draw (0) to (2);
		\draw (0) to (3);
		\draw (1) to (4);
		\draw (2) to (5);
		\draw (7.center)
			 to [in=-90, out=-90, looseness=1.25] (6.center)
			 to [in=90, out=90, looseness=1.75] cycle;
	\end{pgfonlayer}
\end{tikzpicture}}\]
The drawing of the bubble represents a 2-cell $f_0\Rightarrow f'$.
We can think of this 2-cell, being directed from $f_0$ to $f'$, as "reaching into $f_0$", pulling observable wires out to expose them as the wires of $f'$.

We can compose $f$ with some factor $g$ (with a compatible interface), repurposing the black dot --- or \textit{spider}, in the terminology of \textcite{Coecke2016Categorical} --- to represent the locus of composition:
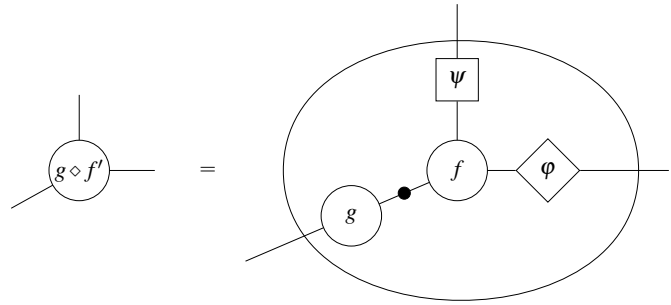
\[\scalebox{0.8}{\begin{tikzpicture}
	\begin{pgfonlayer}{nodelayer}
		\node [style=dot-1cm] (0) at (0, 0) {$f$};
		\node [style=box-7mm] (1) at (0, 3) {$\psi$};
		\node [style=dia-7mm] (2) at (3, 0) {$\varphi$};
		\node [style=dot-1cm] (3) at (-3.5, -1.5) {$g$};
		\node [style=none] (4) at (0, 5.5) {};
		\node [style=none] (5) at (7, 0) {};
		\node [style=none] (6) at (-5.75, 0) {};
		\node [style=none] (7) at (6, 0) {};
		\node [style=none] (8) at (-7, -3) {};
		\node [style=bcopy] (9) at (-1.75, -0.75) {};
		\node [style=dot-1cm] (10) at (-12.5, 0) {$g\diamond f'$};
		\node [style=none] (11) at (-8.25, 0) {$=$};
		\node [style=none] (12) at (-14.75, -1.25) {};
		\node [style=none] (13) at (-12.5, 2.5) {};
		\node [style=none] (14) at (-10, 0) {};
	\end{pgfonlayer}
	\begin{pgfonlayer}{edgelayer}
		\draw (0) to (1);
		\draw (0) to (2);
		\draw (1) to (4.center);
		\draw (2) to (5.center);
		\draw (7.center)
			 to [in=-90, out=-90, looseness=1.25] (6.center)
			 to [in=90, out=90, looseness=1.25] cycle;
		\draw (8.center) to (3);
		\draw (9) to (3);
		\draw (9) to (0);
		\draw (13.center) to (10);
		\draw (10) to (12.center);
		\draw (10) to (14.center);
	\end{pgfonlayer}
\end{tikzpicture}}\]
(In this way, we can think of a one-legged spider, representing an observable variable, as a locus of \textit{potential} composition.)

We can reach inside $g\diamond f'$ to expose the internal wire:
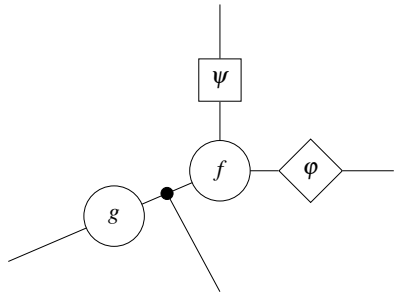
\[\scalebox{0.8}{\begin{tikzpicture}
	\begin{pgfonlayer}{nodelayer}
		\node [style=dot-1cm] (0) at (0, 0) {$f$};
		\node [style=box-7mm] (1) at (0, 3) {$\psi$};
		\node [style=dia-7mm] (2) at (3, 0) {$\varphi$};
		\node [style=dot-1cm] (3) at (-3.5, -1.5) {$g$};
		\node [style=none] (4) at (0, 5.5) {};
		\node [style=none] (5) at (5.75, 0) {};
		\node [style=none] (8) at (-7, -3) {};
		\node [style=bcopy] (9) at (-1.75, -0.75) {};
		\node [style=none] (10) at (0, -4) {};
	\end{pgfonlayer}
	\begin{pgfonlayer}{edgelayer}
		\draw (0) to (1);
		\draw (0) to (2);
		\draw (1) to (4.center);
		\draw (2) to (5.center);
		\draw (8.center) to (3);
		\draw (9) to (3);
		\draw (9) to (0);
		\draw (9) to (10.center);
	\end{pgfonlayer}
\end{tikzpicture}}\]
And this gives us another way to understand the deterministic morphisms: as those that can slide past spiders, multiplying on the way into factors.
For example, suppose have another composite factor $h'$ involving $\varphi$, such as this:
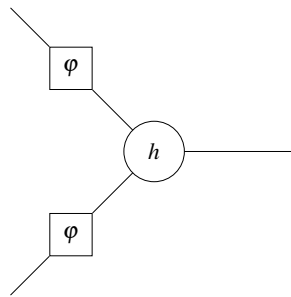
\[\scalebox{0.8}{\begin{tikzpicture}
	\begin{pgfonlayer}{nodelayer}
		\node [style=dot-1cm] (0) at (0, 0) {$h$};
		\node [style=box-7mm] (1) at (-2.75, 2.75) {$\varphi$};
		\node [style=box-7mm] (2) at (-2.75, -2.75) {$\varphi$};
		\node [style=none] (3) at (4.75, 0) {};
		\node [style=none] (4) at (-4.75, 4.75) {};
		\node [style=none] (5) at (-4.75, -4.75) {};
	\end{pgfonlayer}
	\begin{pgfonlayer}{edgelayer}
		\draw (0) to (1);
		\draw (0) to (2);
		\draw (0) to (3.center);
		\draw (1) to (4.center);
		\draw (2) to (5.center);
	\end{pgfonlayer}
\end{tikzpicture}}\]
Then we could compose along $\varphi$ with $f'$ to obtain $h'\diamond f'$:
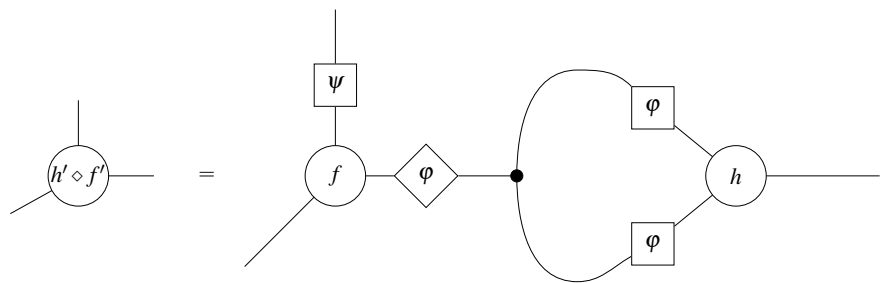
\[\scalebox{0.8}{\begin{tikzpicture}
	\begin{pgfonlayer}{nodelayer}
		\node [style=dot-1cm] (0) at (-2, 0) {$f$};
		\node [style=box-7mm] (1) at (-2, 3) {$\psi$};
		\node [style=dia-7mm] (2) at (1, 0) {$\varphi$};
		\node [style=none] (3) at (-5, -3) {};
		\node [style=none] (4) at (-2, 5.5) {};
		\node [style=bcopy] (5) at (4, 0) {};
		\node [style=dot-1cm] (6) at (11.25, 0) {$h$};
		\node [style=box-7mm] (7) at (8.5, 2.25) {$\varphi$};
		\node [style=box-7mm] (8) at (8.5, -2.25) {$\varphi$};
		\node [style=none] (9) at (16, 0) {};
		\node [style=none] (10) at (6, 3.5) {};
		\node [style=none] (11) at (6, -3.5) {};
		\node [style=dot-1cm] (12) at (-10.5, 0) {$h'\diamond f'$};
		\node [style=none] (13) at (-6.25, 0) {$=$};
		\node [style=none] (14) at (-12.75, -1.25) {};
		\node [style=none] (15) at (-10.5, 2.5) {};
		\node [style=none] (16) at (-8, 0) {};
	\end{pgfonlayer}
	\begin{pgfonlayer}{edgelayer}
		\draw (0) to (1);
		\draw (0) to (2);
		\draw (0) to (3.center);
		\draw (1) to (4.center);
		\draw (2) to (5);
		\draw (6) to (7);
		\draw (6) to (8);
		\draw (6) to (9.center);
		\draw (8)
			 to [in=0, out=-150] (11.center)
			 to [in=-90, out=180] (5.center)
			 to [in=-180, out=90] (10.center)
			 to [in=135, out=0] (7);
		\draw (15.center) to (12);
		\draw (12) to (14.center);
		\draw (12) to (16.center);
	\end{pgfonlayer}
\end{tikzpicture}}\]
Notice that here we have made use of the ability (but not necessity!) of copy-composition to couple multiple variables together.

If again we expose the composed variable, we can pull $\varphi$ outside, too:
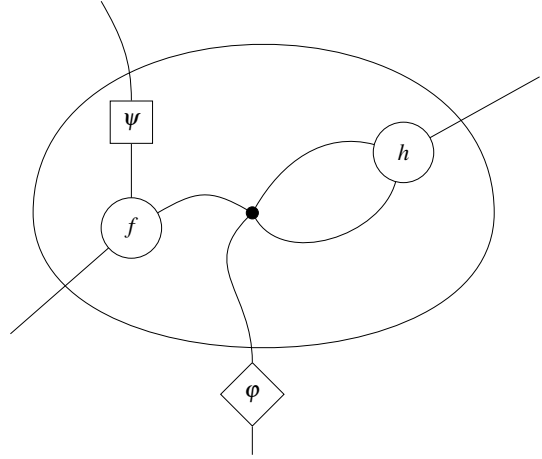
\[\scalebox{0.8}{\begin{tikzpicture}
	\begin{pgfonlayer}{nodelayer}
		\node [style=dot-1cm] (0) at (0, -0.5) {$f$};
		\node [style=box-7mm] (1) at (0, 3) {$\psi$};
		\node [style=none] (3) at (-4, -4) {};
		\node [style=none] (4) at (-1, 7) {};
		\node [style=bcopy] (5) at (4, 0) {};
		\node [style=dot-1cm] (6) at (9, 2) {$h$};
		\node [style=none] (9) at (13.5, 4.5) {};
		\node [style=none] (11) at (-3.25, 0) {};
		\node [style=none] (12) at (12, 0) {};
		\node [style=none] (13) at (4, -8) {};
		\node [style=dia-7mm] (14) at (4, -6) {$\varphi$};
	\end{pgfonlayer}
	\begin{pgfonlayer}{edgelayer}
		\draw (0) to (1);
		\draw (0) to (3.center);
		\draw [in=-60, out=90] (1) to (4.center);
		\draw (6) to (9.center);
		\draw [in=150, out=30, looseness=1.25] (0) to (5);
		\draw [in=-105, out=-60] (5) to (6);
		\draw (12.center)
			 to [in=90, out=90, looseness=1.25] (11.center)
			 to [in=-90, out=-90] cycle;
		\draw [in=90, out=-135, looseness=1.25] (5) to (14);
		\draw (14) to (13.center);
		\draw [in=165, out=60] (5) to (6);
	\end{pgfonlayer}
\end{tikzpicture}}\]
Notice that, by using spiders, there would be no ambiguity if we omitted the outer bubble: observed variables are those that are not terminated in a one-legged spider; and deterministic transformations ``push forward'' from exposed variables to multiply through spiders in the direction of their target factors.

We can use non-deterministic morphisms to render observable variables \textit{unobservable}, by `marginalizing' them out.
For instance, here we have marginalized out one of the observed variables of $h'$ by transforming it with a state $\nu$:
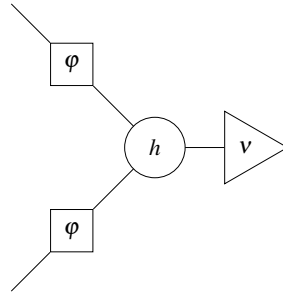
\[\scalebox{0.8}{\begin{tikzpicture}
	\begin{pgfonlayer}{nodelayer}
		\node [style=dot-1cm] (0) at (0, 0) {$h$};
		\node [style=box-7mm] (1) at (-2.75, 2.75) {$\varphi$};
		\node [style=box-7mm] (2) at (-2.75, -2.75) {$\varphi$};
		\node [style=state270] (3) at (3, 0) {$\nu$};
		\node [style=none] (4) at (-4.75, 4.75) {};
		\node [style=none] (5) at (-4.75, -4.75) {};
	\end{pgfonlayer}
	\begin{pgfonlayer}{edgelayer}
		\draw (0) to (1);
		\draw (0) to (2);
		\draw (0) to (3);
		\draw (1) to (4.center);
		\draw (2) to (5.center);
	\end{pgfonlayer}
\end{tikzpicture}}\]
Here, we have used the standard string-diagrammatic representation of a state in $\Ca$.

Finally, let us depict the factor graph of Figure~\ref{fig:fg-1} in this language:
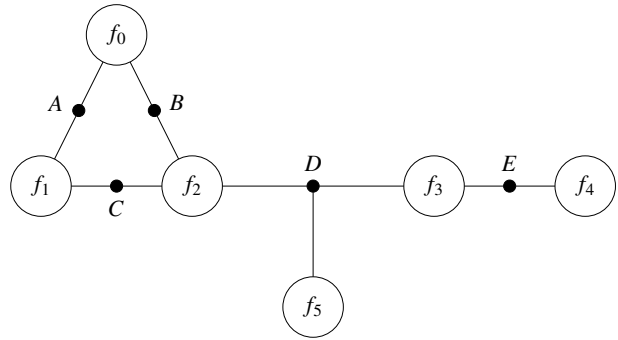
\[\scalebox{0.8}{\begin{tikzpicture}
	\begin{pgfonlayer}{nodelayer}
		\node [style=dot-1cm] (0) at (-5, 2) {$f_1$};
		\node [style=dot-1cm] (1) at (0, 2) {$f_2$};
		\node [style=dot-1cm] (2) at (8, 2) {$f_3$};
		\node [style=dot-1cm] (3) at (4, -2) {$f_5$};
		\node [style=bcopy] (4) at (4, 2) {};
		\node [style=none] (5) at (-2.5, 1.25) {$C$};
		\node [style=none] (6) at (4, 2.75) {$D$};
		\node [style=dot-1cm] (9) at (-2.5, 7) {$f_0$};
		\node [style=dot-1cm] (10) at (13, 2) {$f_4$};
		\node [style=none] (11) at (-4.5, 4.75) {$A$};
		\node [style=none] (12) at (10.5, 2.75) {$E$};
		\node [style=none] (13) at (-0.5, 4.75) {$B$};
		\node [style=bcopy] (14) at (-1.25, 4.5) {};
		\node [style=bcopy] (15) at (-3.75, 4.5) {};
		\node [style=bcopy] (16) at (-2.5, 2) {};
		\node [style=bcopy] (17) at (10.5, 2) {};
	\end{pgfonlayer}
	\begin{pgfonlayer}{edgelayer}
		\draw (1) to (4);
		\draw (4) to (2);
		\draw (3) to (4);
		\draw (9) to (14);
		\draw (14) to (1);
		\draw (0) to (15);
		\draw (15) to (9);
		\draw (0) to (16);
		\draw (16) to (1);
		\draw (2) to (17);
		\draw (17) to (10);
	\end{pgfonlayer}
\end{tikzpicture}}\]
Note how similar this depiction is to the non-compositional form of Figure~\ref{fig:fg-1}: the major change is the replacement of `variable' nodes with (labelled) two-legged spiders, making the distinction between observed and observable variables.
Beyond this distinction, the syntax of $\FG$ tells us precisely what kinds of compositions are allowed --- including supplying notions of \textit{transformation} of factors.

\end{document}